\documentclass[11pt]{article}

\RequirePackage[OT1]{fontenc}
\RequirePackage{amsthm,amsmath}
\RequirePackage[numbers]{natbib}
\RequirePackage[colorlinks,citecolor=blue,urlcolor=blue]{hyperref}
\RequirePackage{hypernat}

\usepackage{amssymb}
\usepackage{enumerate}
\usepackage{color}

\numberwithin{equation}{section}
\theoremstyle{plain}
\newtheorem{thm}{Theorem}[section]
\theoremstyle{remark}
\newtheorem{rem}{Remark}[section]
\theoremstyle{definition}

\theoremstyle{plain}
\newtheorem{pro}{Proposition}[section]
\theoremstyle{plain}
\newtheorem{lem}{Lemma}[section]
\theoremstyle{plain}
\newtheorem{cor}{Corollary}[section]
\newtheorem{ass}{AS\hspace*{-3pt}}

\def\r{\rightarrow}
\newcommand{\E}{\mathbb{E}}     
\renewcommand{\P}{\mathbb{P}}     
\renewcommand{\L}{\mathbb{L}}     
\newcommand{\F}{\mathbb{F}}     
\newcommand{\N}{\mathbb{N}}     

\newcommand{\R}{\mathbb{R}}     
     
\newcommand{\C}{\mathbb{C}} 
 
\newcommand{\X}{\mathbb{X}} 
\newcommand{\T}{\mathbb{T}} 

\newcommand{\cA}{\mbox{$\cal A$}}
\newcommand{\cB}{\mbox{$\cal B$}}
\newcommand{\cC}{\mbox{$\cal C$}}
\newcommand{\cF}{\mbox{$\cal F$}}
\newcommand{\cG}{\mbox{$\cal G$}}
\newcommand{\cL}{\mbox{$\cal L$}}
\newcommand{\cM}{\mbox{$\cal M$}}
\newcommand{\cN}{\mbox{$\cal N$}}
\newcommand{\cP}{\mbox{$\cal P$}}
\newcommand{\cX}{\mbox{$\cal X$}}
\newcommand{\cV}{\mbox{$\cal V$}}
\newcommand{\cO}{\mbox{$\cal O$}}

\newcommand{\mprod}[2]{\overset{#2}{\underset{i=#1}{\otimes}}}

\begin{document}

\title{Limit theorems for stationary Markov processes with $\L^2$-spectral gap}  

\author{Déborah FERR\'E, Loïc HERV\'E, James LEDOUX  \footnote{
Université
Européenne de Bretagne, I.R.M.A.R. (UMR-CNRS 6625), Institut National des Sciences 
Appliquées de Rennes. Deborah.Ferre,Loic.Herve,JamesLedoux@insa-rennes.fr}}

\maketitle

\begin{abstract}
 Let $(X_t,Y_t)_{t\in\T}$ be a discrete or continuous-time Markov process with state space $\X\times \R^d$ where $\X$ is an arbitrary measurable set. Its transition semigroup is assumed to be additive with respect to the second component, i.e.~$(X_t,Y_t)_{t\in\T}$ is assumed to be a Markov additive  process. In particular, this implies that the first component $(X_t)_{t\in\T}$ is also a Markov process. Markov random walks or additive functionals of a Markov process are special instances of Markov additive processes. 
 In this paper, the process $(Y_t)_{t\in\T}$ is shown to satisfy the following classical limit theorems: 
\begin{enumerate}[(a)]
\item the central limit theorem, 
\item the local limit theorem, 
\item the one-dimensional Berry-Esseen theorem, 
\item the one-dimensional first-order Edgeworth expansion, 
\end{enumerate}
provided that we have $\sup_{t\in(0,1]\cap\T}\E_{\pi,0}[|Y_t|^{\alpha}]<\infty$ with the expected order $\alpha$ with respect to the independent case (up to some $\varepsilon>0$ for (c) and (d)). For the statements (b) and (d), a Markov nonlattice condition is also assumed as in the independent case. 
 All the results are derived under the assumption that the Markov process $(X_t)_{t\in\T}$ has an invariant probability distribution $\pi$, is stationary and has the $\L^2(\pi)$-spectral gap property (that is,  $(X_t)_{t\in\N}$ is $\rho$-mixing in the discrete-time case). The case where  $(X_t)_{t\in\T}$ is non-stationary is briefly discussed. As an application, we derive a Berry-Esseen bound for the $M$-estimators associated with $\rho$-mixing Markov chains.   \\[2mm]
\end{abstract}

\section{Introduction} 

In this paper, we are concerned with the class of  Markov Additive Processes (MAP). The discrete and continuous-time cases are considered so that the time parameter set $\T$ will denote $\N$ or $[0,+\infty)$. Let $\X$ be any set equipped by a 
 $\sigma$-algebra $\cX$ and let $\cB(\R^d)$ be the Borel $\sigma$-algebra on $\R^d$ ($d\ge 1$).  A (time homogeneous) MAP $(X_t,Y_t)_{t\in \T}$ is a (time homogeneous) Markov process with state space $\X\times \R^d$ and transition semigroup $(Q_t)_{t\in\T}$ satisfying: $\forall t \in \T$, $\forall (x,y)\in \X\times \R^d$, $\forall (A,B)\in \cX \times \cB(\R^d)$,
\begin{equation} \label{Add_pro}
	 Q_t(x,y;A\times B) = Q_t(x,0;A\times B-y).
\end{equation}
In other words, the transition semigroup is additive in the second component. It follows from the definition that the first component $(X_t)_{t\in \T}$ of a MAP  is a (time homogeneous) Markov process. The second component $(Y_t)_{t\in \T}$ must be thought of as a process with  independent increments given $\sigma(X_s, s \ge 0)$. We refer to \cite{Cin72-2} for the general structure of such processes. Note that a discrete-time MAP is also called a Markov Random Walk (MRW). In stochastic modelling, the first component of a MAP is usually associated with a random environment which drives or modulates the additive component $(Y_t)_{t\in \T}$. The MAPs have been found to be an important tool in various areas as communication networking (e.g. see \cite{PacTanPra09,PacPra95,Asm03}), finance (e.g. see \cite{Asm00,AsmAvrPis04,JobRog06}), reliability (e.g. see \cite{Cin77,OzeSoy04,LimOpr01,GraLed04}), \ldots Some important instances of MAP  are: 
\begin{itemize}
\item in discrete/continuous-time : $(X_t,Y_t)_{t\in \T}$ where $(Y_t)_{t\in\T}$ is a $\R^d$-valued additive functional (AF) of the Markov process $(X_t)_{t\in\T}$. Therefore any result on the second component of a MAP applies to an AF.   
Basic discrete and continuous-time AFs are respectively 
\begin{equation} \label{Basic_FA}
	Y_0=0, \, \forall t\in\N^*, \quad Y_t = \sum_{k=1}^{t} \xi(X_k); \qquad \forall t \in [0,+\infty[, \quad Y_t = \int_{0}^{t} \xi(X_s) \, ds
\end{equation}
where $\xi$ is a $\R^d$-valued function satisfying conditions under which $Y_t$ is well-defined for every $t\in \T$. When $(X_t)_{t\in\T}$ is a regular Markov jump process, then any non-decreasing AF has the form (e.g. \cite{Cin75})
\begin{equation*} 
	\int_{0}^{t} \xi_1(X_s) \, ds + \sum_{s\le t} \xi_2(X_{s-},X_s)
\end{equation*}
where $X_{t-}= \lim_{s\rightarrow t, s<t} X_s$, $\xi_1$ and $\xi_2$ are non-negative measurable functions such that $\xi_2(x,x)=0$ for every $x\in \X$. General representations and properties of AFs may be found in \cite[and references therein]{BenJac73,RevYor99}.  
 Such AFs are basically introduced when some kind of ``rewards'' are collected along with the dynamics of the Markov process $(X_t)_{t\in\T}$ through the state space $\X$. Thus, $Y_t$ is the accumulated reward  on the finite interval $[0,t]$.  
Even if the state space $\X$ is a finite set, the numerical computation of the probability distribution of such AFs is not an easy task (e.g. see \cite{BlaMeiNeuSer02,Ste05}). 

	\item in discrete-time: the Markov renewal processes when the random variables $Y_t$, $t\in\N$, are non-negative; if we consider a hidden Markov chain $(X_t,Z_t)_{t\in \N}$, where the so-called observed process $(Z_t)_{t\in \N}$ is $\R^d$-valued ($Z_0=0$), then $(X_t,\sum_{k=1}^{t} Z_k)_{t\in \N}$ is a MAP.
	\item in continuous time: the Markovian Arrival Process where $(X_t)_{t\in \T}$ is a regular jump process and $(Y_t)_{t\in \T}$ is a point process (see \cite{Asm03}), which includes the so-called Markov Modulated Poisson Process. 
\end{itemize}

 Seminal works on MAPs are \cite{Nev61,EzhSko69-1,EzhSko69-2,KeiWis64,Pin68} and are essentially concerned with a finite Markov process $(X_t)_{t\in \T}$ as first component. The second component $(Y_t)_{t\in \T}$ was sometimes called a process defined on a Markov process.  When $\X$ is a finite set, the structure of MAPs are well understood and an account of what is known can be found in \cite[Chap~XI]{Asm03}. 
In this paper, we are concerned with Gaussian approximations of the distribution of the second component $Y_t$ of a MAP. Central limit theorems  for $(Y_t)_{t\in \T}$ may be found in \cite{KeiWis64,Pin68,FukHit67,HitShi70,Bha82,Tou83,KipVar86,GlyWhi93,Ste01,Hol05} under various assumptions. 
 Here, such results are derived when $(X_t)_{t\in\T}$ has an invariant probability measure $\pi$, is stationary and has the $\L^2(\pi)$-spectral gap property (see conditions (AS\ref{AS1}-AS\ref{AS-spectralgap}) below). Moreover, standard refinements of the central limit theorem (CLT) related to the convergence rate are provided. Before, notations and assumptions used throughout the paper are introduced. 

Let $(X_t,Y_t)_{t\in \T}$ be a MAP with state space $\X\times \R^d$ and transition semigroup $(Q_t)_{t\in\T}$. $(\X,\cX)$ is assumed to be a measurable space equipped with a  
$\sigma$-algebra $\cX$. In the continuous-time case, $(X_t,Y_t)_{t\in \T}$ is assumed to be progressively measurable. $(X_t)_{t\in\T}$ is also a Markov process with transition semigroup $(P_t)_{t\in\T}$ given by  
\begin{equation*}
	P_t(x,A) := Q_t(x,0; A\times \R^d).
\end{equation*}
 Throughout the paper, we assume that $(X_t)_{t\in\T}$ has a unique invariant probability measure denoted by $\pi$ $(\forall t\in\T,\ \pi\circ P_t = \pi)$. We denote by $\L^2(\pi)$ the usual Lebesgue space of (classes of) functions $f : \X \r \C$ such that $\|f\|_2:=\sqrt{\pi(|f|^2)} = (\int_\X |f|^2 d\pi)^{1/2} <\infty$. The operator norm of a bounded linear operator $T$ on $\L^2(\pi)$ is defined by $\|T\|_2:=\sup_{\{f\in\L^2(\pi):\|f\|_2=1\}}\|T(f)\|_2$. We appeal to the following conditions.
\begin{ass} \label{AS1}
 $(X_t)_{t\in\T}$ is stationary (i.e.~$X_0\sim \pi$).
\end{ass}
\begin{ass} \label{AS-spectralgap} The semigroup $(P_t)_{t\in\T}$ of $(X_t)_{t\in\T}$ has a spectral gap on $\L^2(\pi)$:
\begin{equation} \label{strong-erg1}
\lim_{t\r+\infty}\|P_t-\Pi\|_2=0, 
\end{equation}
where $\Pi$ denotes the rank-one projection defined on $\L^2(\pi)$ by: $\Pi f = \pi(f)1_{\X}$.  
\end{ass}
\begin{ass} \label{AS-alpha} The process $(Y_t)_{t\in \T}$ satisfies the moment condition  
\begin{equation} \label{moment-alpha} 
\sup_{t\in(0,1]\cap\T}\E_{\pi,0}[|Y_t|^{\alpha}]<\infty  
\end{equation}
where $|\cdot|$ denotes the euclidean norm on $\R^d$ and $\E_{\pi,0}$ is the expectation when  $(X_0,Y_0)\sim (\pi,\delta_0)$.
\end{ass}
In the discrete-time case, notice that the moment condition~(\ref{moment-alpha}) reduces to \textbf{(AS3d)}
\begin{equation} \label{AS3d} 
 \E_{\pi,0}[|Y_1|^{\alpha}]<\infty  \tag{AS3d}
\end{equation}
and that condition (AS\ref{AS-spectralgap}) is  equivalent to the $\rho$-mixing property of $(X_t)_{t\in\N}$, with $\rho$-mixing coefficients going to 0 exponentially fast \cite{Ros71}. Condition (AS\ref{AS-spectralgap}) is also related to the notion of essential spectral radius (e.g.~see \cite{Wu04}).

Under (AS\ref{AS1}-AS\ref{AS-spectralgap}), we show that the second component $(Y_t)_{t\in\T}$ of the MAP satisfies, in discrete and continuous time, the following standard limit theorems : {\it 
\begin{enumerate}[(a)] \setlength{\itemsep}{0mm}
\item the central limit theorem, under (AS\ref{AS-alpha}) with the optimal value $\alpha=2$;
\item the local limit theorem, under (AS\ref{AS-alpha}) with the optimal value $\alpha=2$ and the additional classical Markov non-lattice condition;   
\item the one-dimensional Berry-Esseen theorem, under (AS\ref{AS-alpha}) with the (almost) optimal value  ($\alpha>3$);
\item  a one-dimensional first-order Edgeworth expansion, under (AS\ref{AS-alpha}) with the (almost) optimal value ($\alpha>3$) and the Markov non-lattice condition.
\end{enumerate}
}
These results correspond to the classical statements for the sequences of independent and identically distributed (i.i.d.) random variables, with the same order $\alpha$ (up to $\varepsilon>0$ in (c) and (d)). Such results are known for special MAPs satisfying (AS\ref{AS-spectralgap}) (comparison with earlier works is made after each statement), but to the best of our knowledge, the results (a)-(d) are new for general MAPs satisfying (AS\ref{AS-spectralgap}), as, for instance, for AF involving unbounded functionals. 

Here, the main arguments are 
\begin{itemize} \setlength{\itemsep}{0mm}
	\item for the statement (a): the $\rho$-mixing property  of the increments $(Y_{t+1}-Y_{t})_{t\in\T}$ of the process $(Y_t)_{t\in\T}$ (see Proposition~\ref{mixing}).  This result, which has its own interest, is new to the best of our knowledge. The closest work to this part is a result of \cite{GriOpr76} which, by using $\phi$-mixing properties, gives the CLT for MAPs associated with uniformly ergodic driving Markov chains (i.e.~$(P_t)_{t\in\T}$  has a spectral gap on the usual Lebesgue space $\L^\infty(\pi)$). Condition~(AS\ref{AS-spectralgap}) is less restrictive than uniform ergodicity (which is linked to the so-called Doeblin condition).  
\item For the refinements (b-d) : the Nagaev-Guivarc'h spectral method. The closest works to this part are,  in discrete-time the paper  \cite{HerPen09} in which these refinements are obtained for the AF: $Y_t = \sum_{k=1}^t\xi(X_k)$, and in continuous-time the work of Lezaud \cite{Lez01} which proves, under the uniform ergodicity assumption, a Berry-Esseen bound for the integral additive functional (\ref{Basic_FA}). Here, in discrete-time, we borrow to a large extent the weak spectral method of \cite{HerPen09}: this is outlined in Proposition~\ref{carac-en-0-disc}, which gives a precise expansion (close to the i.i.d.~case) of  the characteristic function of $Y_t$. For continuous-time MAPs, similar expansions can be derived from the semigroup property of the Fourier operators of the MAP. Proposition~\ref{carac-en-0-disc},  and its continuous-time counterpart Proposition~\ref{carac-en-0-cont}, are the key results to establish limit theorems (as for instance the statements (b-d)) with the help of Fourier techniques. 
\end{itemize}

The classical (discrete and continuous-time) models for which the spectral gap property~(AS\ref{AS-spectralgap}) is met, are briefly reviewed in   Subsections~\ref{sub-sec-geo-ergo}-\ref{sub-sec-cont-time-ex}.
 The above limit theorems (a)-(d) are valid in all these examples and open up possibilities for new applications. First, our moment conditions are optimal (or almost optimal). For instance, in continuous time, the Berry-Esseen bound in  \cite{Lez01} requires that $\xi$ in the integral (\ref{Basic_FA}) is bounded, while our statement (c) holds true under the condition  $\pi(|\xi|^{3+\varepsilon})<\infty$. Second, our results are true for general MAPs. For instance, they apply to $Y_t = \sum_{k=1}^t \xi(X_{k-1},X_k)$. This fact enables us to prove a Berry-Esseen bound for $M$-estimators associated with  $\rho$-mixing Markov chains, under a moment condition which greatly improves the results in \cite{Rao73}.
 
The paper is organised as follows. The $\L^2(\pi)$-spectral gap assumption for a Markov process is briefly discussed in Section~\ref{Section_L2} and connections to standard ergodic properties are pointed out. In Section~\ref{Section_CLT}, the CLT for $(Y_t)_{t\in\T}$ under (AS\ref{AS1})-(AS\ref{AS-alpha}) with $\alpha=2$ is derived. The functional central limit theorem (FCLT) is also discussed. Section~\ref{Section_Rafinements} is devoted to refinements of the CLT.  
First, the Fourier operator is introduced in Subsection~\ref{SS_Fourier}, the characteristic function of $Y_t$ is investigated  in Subsection~\ref{SS_expansions}, and our limit theorems are proved for discrete-time MAPs in Subsection~\ref{SS_DT}. Their extension to the non-stationary case is discussed in Subsection~\ref{non-stat-carac}. The continuous-time case is studied in Subsection~\ref{SS_CT}. 
The statistical application to $M$-estimators for $\rho$-mixing Markov chains is developed in Section~\ref{sect-B-E-statist}. 

Finally, we point out that the natural way to consider the Nagaev-Guivarc'h method in continuous-time is the  semigroup property of the Fourier operators of the MAP (see Subsection~\ref{SS_Fourier} for details). To the best of our knowledge, this property, which is  closely related to the additivity condition (\ref{Add_pro}) defining a MAP, has been introduced and only exploited in \cite{HitShi70}.

\section{The $\L^2(\pi)$-spectral gap property (AS\ref{AS-spectralgap})} \label{Section_L2}

\subsection{Basic facts on property (AS\ref{AS-spectralgap})} \label{sec-basic-fact-map}

We discuss the condition~(AS\ref{AS-spectralgap}) for the semigroup $(P_t)_{t\in\T}$ of $(X_t)_{t\in\T}$.  
It is well-known that $(P_t)_{t\in\T}$ is a contraction semigroup on each Lebesgue-space $\L^p(\pi)\, $ for $1\leq p \leq +\infty$, that is: we have $\|P_t\|_p \leq 1$ for all $t\in\T$, where $\|\cdot\|_p$ denotes the operator norm on $\L^p(\pi)$. 
Condition~(AS\ref{AS-spectralgap}), introduced by Rosenblatt \cite{Ros71} and also called strong ergodicity on $\L^2(\pi)$, implies that $(P_t)_{t\in\T}$ is strongly ergodic on each $\L^p(\pi)\, $ ($1< p < +\infty$), that is $\|P_t-\Pi\|_p \r 0$ when $t\r +\infty$. Moreover, (AS\ref{AS-spectralgap}) is fulfilled under the so-called uniform ergodicity property, i.e.~the strong ergodicity on $\L^\infty(\pi)$. These properties, established in \cite{Ros71}, can be easily derived from the Riesz-Thorin interpolation theorem \cite{BerLof76}  which insures, thanks to the contraction property of $P_t$, that 
\begin{equation} \label{rev-interpol-ineg}
\|P_t-\Pi\|_p \leq \|P_t-\Pi\|_{p_1}^{\alpha} \|P_t-\Pi\|_{p_2}^{1-\alpha} \leq 2\, \min\big\{\|P_t-\Pi\|_{p_1}^{\alpha}, \|P_t-\Pi\|_{p_2}^{1-\alpha}\big\},
\end{equation}
where $p_1,p_2\in[1,+\infty]$ and $p\in[1,+\infty]$ satisfy $1/p = \alpha/p_1 + (1-\alpha)/p_2$ for some $\alpha\in[0,1]$.  
Indeed, assume that Condition~(AS\ref{AS-spectralgap}) holds. Then Inequality~(\ref{rev-interpol-ineg}) with $(p_1=2,p_2=+\infty)$ and  $\alpha\in(0,1)$ gives the strong ergodicity on $\L^p(\pi)$ for each $p\in (2,+\infty)$. Notice that the value $p=+\infty$ is obtained with $\alpha=0$, but in this case, the uniform ergodicity cannot be deduced from (AS\ref{AS-spectralgap}) and (\ref{rev-interpol-ineg}). In fact the uniform ergodicity condition is stronger than (AS\ref{AS-spectralgap}) (see \cite{Ros71}). Next Inequality~(\ref{rev-interpol-ineg}) with $(p_1=2,p_2=1)$ and  $\alpha\in(0,1)$ gives the strong ergodicity on $\L^p(\pi)$ for each $p\in (1,2)$. The value $p=1$ is obtained with $\alpha=0$, but the strong ergodicity on $\L^1(\pi)$ cannot be deduced from (AS\ref{AS-spectralgap}) and (\ref{rev-interpol-ineg}). Finally, if the uniform ergodicity is assumed, then Inequality~(\ref{rev-interpol-ineg}) with $(p_1=+\infty,p_2=1)$ and  $\alpha=1/2$ yields (AS\ref{AS-spectralgap}). 

Also notice that the strong ergodicity property on $\L^p(\pi)$ holds if and only if there exists some strictly positive constants $C$ and $\varepsilon$ such that we have for all $t\in\T$: 
\begin{equation} \label{Expo}
\|P_t-\Pi\|_p \leq C\, e^{-\varepsilon t}.	
\end{equation}
Indeed, if $\kappa_0 := \|P_\tau-\Pi\|_p <1 $ for some $\tau\in\T$ (which holds under the strong ergodicity property), then we have for all $n\in\N^*$:  
$\|P_{n\tau}-\Pi\|_p = \|P_\tau^n-\Pi\|_p = \|(P_\tau-\Pi)^n\|_p \leq \kappa_0^n$. Writing $t = w+ n\tau$ with $n\in\N^*$ and $w\in[0,\tau)$, we obtain: $\|P_t-\Pi\|_p = \|P_w(P_\tau^n - \Pi)\|_p \leq \kappa_0^n \leq C\, e^{-\varepsilon t}$ with $C:=1/\kappa_0$ and $\varepsilon:=(-1/\tau)\ln\kappa_0$. The converse implication is obvious. Thus, the strong ergodicity property on $\L^2(\pi)$, i.e condition (AS\ref{AS-spectralgap}), is equivalent to require $\L^2(\pi)$-exponential ergodicity (\ref{Expo}), that is the $\L^2(\pi)$-spectral gap property.

In the next subsection, Markov models with a spectral gap on $\L^2(\pi)$ arising  from stochastic modelling and potentially relevant to our framework are introduced. Assumption (AS2) can be also met in more abstract settings, as for 
instance in \cite{Gui02} where the $\L^2$-spectral gap property for classic Markov operators (with a state space defined as the $d$-dimensional torus) is proved.

\subsection{Geometric ergodicity and property (AS\ref{AS-spectralgap})} \label{sub-sec-geo-ergo} 

 Recall that $(X_t)_{t\in \N}$ is $V$-geometrically ergodic if its transition kernel $P$ has an invariant probability measure $\pi$ and is such that there are $r\in(0,1)$, a finite constant K and a $\pi$-a.e finite function $V: \X \mapsto [1,+\infty]$ such that 
	\[  \forall n\ge 0, \ \pi-\text{a.e. } x\in \X, \quad \sup\big\{ |P^nf(x) -\pi(f)|, \ f : \X\r \C, |f|\le V \big\} \le K\, V(x) \, r^n. \tag{VG}
\]
  In fact, when $(X_t)_{t\in \N}$ is $\psi$-irreducible (i.e. $\psi(A)>0 \Longrightarrow P(x,A)>0, \forall x\in\X$) and aperiodic \cite{MeyTwe93}, condition (VG) is equivalent to the standard geometric ergodicity property \cite{RobRos97}: there are functions $r : \X \r (0,1)$ and $C:\X \mapsto [1,+\infty)$ such that: for all $n\in \N,$ $\pi-\text{a.e. } x\in \X$,
	\[ \big\|P^n(x,\cdot) -\pi(\cdot)\big\|_{\mathrm{TV}}:= \sup\big\{ |P^nf(x) -\pi(f)|, \ f : \X\r \C, |f|\le 1 \big\} \le C(x) \, r(x)^n. 
\]
There is another equivalent operational condition to geometric ergodicity for  $\psi$-irreducible and aperiodic Markov chains $(X_t)_{t\in \N}$, the so-called `` drift-criterion'': there exist a function $V: \X \r [1,+\infty]$, a small set $C\subset \X$ and constants $\delta>0, b<\infty$ such that
	\[ P V \le (1-\delta) V + b 1_C.
\]
 We refer to \cite{MeyTwe93} for details and applications, and to \cite{Jon04} for a recent survey on the CLT for the additive functionals of $(X_t)_{t\in\N}$ in (\ref{Basic_FA}). 
Now, the transition kernel $P$ is said to be reversible with respect to $\pi$ if 
	\[ \pi(dx) P(x,dy) = \pi(dy) P(y,dx)
\]
or equivalently if $P$ is self-adjoint on the space $\L^2(\pi)$. It is well known that a $V$-geometrically ergodic Markov chain with a reversible transition kernel has the $\L^2(\pi)$-spectral gap property \cite{RobRos97}. Moreover, 
 for a $\psi$-irreducible and aperiodic Markov chain $(X_t)_{t\in \N}$ with reversible transition kernel, ($V$-)geometric ergodicity is shown to be equivalent to the existence of a spectral gap in $\L^2(\pi)$, and, when $X_0\sim \mu$, we also have \cite[Th~2.1]{RobRos97},\cite{RobTwe01}
\[ \big\|\mu P^n(\cdot) -\pi(\cdot)\big\|_{\mathrm{TV}} \le \frac{1}{2} \big|\mu -\pi\big|_{L^2(\pi)} \, r^n. \tag{R}
\]
 where $r:=\lim_{n\r +\infty}\big(\| P^n - \Pi\|_2\big)^{1/n}$ and $\big|\mu -\pi\big|_{L^2(\pi)}:=\| d\mu/d\pi -1\|_{2}$ if well-defined and $\infty$ otherwise.
 Note that the reversibility condition is central to the previous discussion on the $\L^2(\pi)$-spectral gap property. Indeed, there exists a $\psi$-irreducible and aperiodic Markov chain which is geometrically ergodic but does not admit a spectral gap on $\L^2(\pi)$ \cite{Hag06}. 
 
Such a context of geometric ergodicity and reversible kernel is relevant to the Markov Chain Monte Carlo methodology for sampling a given probability distribution, i.e.~the target distribution. Indeed, the basic idea is to define a Markov chain $(X_t)_{t\in \N}$ with the target distribution as invariant probability measure $\pi$. Then a MCMC algorithm  is a scheme to draw samples from the  stationary Markov chain $(X_t)_{t\in \N}$. But, the initial condition of the algorithm, i.e.~the probability distribution of $X_0$, is not $\pi$ since the target distribution is inaccessible. Therefore  the convergence in distribution of the Markov chain  
to $\pi$ in regard of the probability distribution of $X_0$ must be guaranteed  and the knowledge of the convergence rate is crucial to monitor the sampling. Thus,  central limit theorem for the Markov chains and quantitative bounds as in (R) are highly expected. 
  Geometric ergodicity of Hasting-Metropolis type algorithms has been investigated by many researchers. Two standard instances are the full dimensional and random-scan symmetric random walk Metropolis algorithm \cite[and references therein]{JarHan00,ForMouRobRos03}. Note that the first algorithm is also referred to as a special instance of the Hasting algorithm and the second one to as a Metropolis-within-Gibbs sampler. Let $\pi$ be a  probability distribution on $\R^d$ which is assumed to have a positive and continuous density with respect to the Lebesgue measure. The so-called proposal  densities are assumed to be bounded away from $0$ in some region around zero (the moves through the state space $\X$ are based on these probability distributions). These conditions assert that the corresponding transition kernel for each algorithm is $\psi$-irreducible, aperiodic and is reversible with respect to $\pi$.
Geometric ergodicity for the Markov chain $(X_t)_{t\in\N}$ (and so the existence of a spectral gap in $\L^2(\pi)$) is closely related to the tails of the target distribution $\pi$. For instance, in the first algorithm, it can be shown that $\pi$ must have an exponential moment \cite[Cor~3.4]{JarHan00}. A sufficient condition for geometric ergodicity in case of super-exponential target densities, is of the form \cite[Th~4.1]{JarHan00}
	\[ \lim_{|x|\r +\infty} \langle \frac{x}{|x|}, \frac{\nabla\pi(x)}{|\nabla\pi(x)|} \rangle <0.
\]
For the second algorithm, sufficient conditions for geometric ergodicity are reported in \cite{ForMouRobRos03} when the target density decreases either subexponentially or exponentially in the tails. A very large set of examples and their respective merit are discussed in these two references. 
We refer to \cite[and references therein]{RobRos04} for a recent survey on the theory of Markov chains in connection with MCMC algorithms.

\subsection{Uniform ergodicity and hidden Markov chains} \label{sub-sec-unif-ergo} 

As quoted in the introduction, a discrete-time MAP is closely related to a hidden Markov chain. Standard issues for hidden Markov chains require to be aware of the convergence rate of the hidden Markov state process $(X_t)_{t\in\N}$. One of them is the state estimation via filtering or smoothing. 
In such a context, minorization conditions on $P$  are usually involved. The basic one is: 
 there exists a bounded positive measure $\varphi$ on $\X$ such that for some $m\in\N^*$:
	\[ \forall x \in\X, \forall A \in \cX, \quad P^m(x,A)\ge \varphi(A). \tag{UE}
\]
It is well-known that this is equivalent to the uniform ergodicity property or to condition (VG) with $V(x)=1$ \cite[Th 16.2.1, 16.2.2]{MeyTwe93}. Recall that uniform ergodicity gives the $\L^2(\pi)$-spectral gap property (AS\ref{AS-spectralgap}), but the converse is not true. 
 Another minorization condition is the so-called ``Doeblin condition'': there exists a probability measure $\varphi$ such that for some $m$, $\varepsilon <1$ and $\delta >0$ \cite{Doo53}
	\[ \varphi(A) > \varepsilon \Longrightarrow \forall x\in \X, \quad P^m(x,A) \ge \delta. \tag{$D_0$} 
\]
 It is well known that, for ergodic and aperiodic Markov chains, $(D_0)$ is equivalent to the uniform ergodicity. We refer to \cite[and the references therein]{CapMouRyd05} for an excellent overview of the interplay between the Markov chain theory and the hidden Markov models.

\subsection{Property~(AS\ref{AS-spectralgap}) for continuous time Markov processes} \label{sub-sec-cont-time-ex}
The Markov jump processes are a basic class of continuous-time Markov models which has a wide interest in stochastic modelling. The $\L^2(\pi)$-exponential convergence has received attention  a long time ago. We refer to \cite{Che04} for a good account of what is known on ergodic properties for such processes. 
 In particular, the $\L^2(\pi)$-spectral gap property is shown to be equivalent to the standard exponential ergodicity for the birth-death processes:
\begin{equation*}
	\exists \beta >0 \text{ such that }\forall (i,j)\in \X^2, \exists C_{i}\ge 0, \quad |P_t(i,j)-\pi_j | \le C_{i}\exp(-\beta t)  \quad t\rightarrow +\infty 
\end{equation*}
 where $(P_t(i,j))_{i,j\in \X}$ is the matrix semigroup of $(X_t)_{t\ge 0}$. 
This is also true for the reversible Markov jump processes. Hence, in these cases, criteria for exponential ergodicity are also valid to check the $\L^2(\pi)$-exponential convergence. Moreover, explicit bounds on the spectral gap are discussed in details in \cite{Che04}. For the birth-death processes, we also refer  to \cite[and references therein]{Kar00}) where  explicit formulas are obtained for classical Markov queuing processes. The birth-death processes are often used as reference processes for analyzing general stochastic models. This idea was in force in the Liggetts's derivation of the $\L^2$-exponential convergence of  supercritical nearest particle systems \cite{Lig89}. The interacting systems of particles are also a source of examples of processes with a $\L^2$-spectral gap. We refer to \cite{Lig89} for such a discussion on various classes of stochastic Ising models.  
In physics and specially in statistical physics, many evolution models are given by stochastic  ordinary/partial equations. When the solutions are finite/infinite dimensional Markov processes, standard issues arise: existence and uniqueness  of an invariant probability measure, ergodic properties which include the rate of convergence to the invariant measure with respect to some norm. 
Such issues may be included in the general topic of the stability of solutions of stochastic differential equations (SDEs).  
Thus, it is not surprising that ergodic concepts as the $V$-geometric ergodicity and Lyapunov-type criteria associated with, originally developed by Meyn and Tweedie \cite{MeyTwe93} for studying the stability of discrete-time Markov models, have been found to be of value (e.g. see \cite[and references therein]{GolMas06-LN}). Here, we are only concerned with the $\L^2(\pi)$-exponential convergence so that we only mention some results related with. 

An instance of $\L^2(\pi)$-spectral gap can be found in \cite{GanRoySim99} where the following SDE is considered
\begin{equation*}
	dX_t= -\frac{1}{2}\, b(X_t)\, dt + dW_t  \qquad X_0=x\in\R^d
\end{equation*}
where $(W_t)_{t\ge 0}$ is the standard $d$-dimensional Brownian motion and $b(\cdot)$ is a gradient field from $\R^d$ to $\R^d$ (with suitable properties ensuring essentially the existence of a unique strong solution to the equation, which has a unique invariant probability measure). When $b(\cdot)$ is a radial function satisfying $b(x)\sim C |x|^{\alpha}$ for $\alpha >1$ when $x\r +\infty$, then the semigroup is shown to be ultracontractive and to have a $\L^2(\pi)$-spectral gap \cite{GanRoySim99}. 

Another instance of $\L^2(\pi)$-spectral gap is related to the $\R$-valued Markov process solution to 
\begin{equation} \label{Diffusion}
dX_t = b(X_t)\, dt + a(X_t) \, dW_t
\end{equation}
where $(W_t)_{t\ge 0}$ is the standard $1$-dimensional Brownian motion and $X_0$ is a random variable independent of $(W_t)_{t\ge 0}$. Standard assumptions ensure that the solution of the SDE above is a positive recurrent diffusion on some interval and a (strictly) stationary ergodic time-reversible process. Under additional conditions on the scale and the speed densities of the diffusion $(X_t)_{t\ge 0}$  
\cite[(A4) and reinforced (A5), Prop. 2.8]{GenJeaLar00}, the transition semigroup of $(X_t)_{t\ge 0}$ is shown to have the $\L^2(\pi)$-spectral gap property (explicit bounds on the spectral gap are also provided). The basic example studied in \cite{GenJeaLar00} is when $a(x) := c x^{\nu}$ and $b(x):=\alpha(\beta -x)$ with $\nu\in[1/2,1]$, $\alpha,\beta \in\R$. Conditions ensuring the $\L^2(\pi)$- spectral gap property are provided in terms of these parameters. Applications to some classical models in finance are discussed.  
Note that statistical issues for continuous-time Markov processes as the jump or diffusion processes, are related to the time discretization or sampling schemes of these processes. This often provides discrete-time Markov chains which inherit ergodic properties of the original continuous-time process. Thus we turn to the discussion on the discrete-time case (e.g. see \cite{DehYao07} for the jump processes, \cite{GenJeaLar00} and the references therein for the (hidden) diffusions).       
Finally, the context of the stochastic differential equation (\ref{Diffusion}) can be generalized to Markov $H$-valued processes solution to infinite dimensional SDEs, where $H$ is a Hilbert space.
A good account of these generalizations can be found in \cite[and references therein]{GolMas06}.

\section{The $\rho$-mixing property and central limit theorems} \label{Section_CLT}

Let $(X_t,Y_t)_{t\in \T}$ be a MAP taking values in $\X\times\R^d$.  
 $\E_{(x,0)}$, $\E_{\pi,0}$ are the expectation with respect to the initial conditions $(X_0,Y_0)\sim (\delta_{x},\delta_{0})$ and $(X_0,Y_0)\sim (\pi,\delta_{0})$ respectively. First, basic facts for MAPs are proposed.  Second, they are used to show that, for a discrete-time MAP,  the increment process $(Y_n-Y_{n-1})_{n\in\N^*}$ is exponentially $\rho$-mixing under (AS\ref{AS1}-AS\ref{AS-spectralgap}). Then, a CLT is obtained under conditions (AS\ref{AS1}-AS\ref{AS-spectralgap}) and the expected moment condition (AS\ref{AS-alpha}) (i.e.~(\ref{AS3d})) with $\alpha= 2$.  

\subsection{Basic facts on MAPs}

Let $\F^{(X,Y)}_t:=\sigma(X_u,Y_u,\ u\le t)$, $\F^{X}_t:=\sigma(X_u,\ u\le t)$ and $\F^{Y}_t:=\sigma(Y_u,\ u\le t)$ be the filtration generated by the processes $(X_t,Y_t)_{t\in \T}$,  
$(X_t)_{t\in \T}$ and $(Y_t)_{t\in \T}$ respectively. 

The additivity property (\ref{Add_pro}) for the semigroup $(Q_t)_{t\in\T}$ reads as follows for any measurable ($\C$-valued) function $g$ on $\X\times\R^d$ and any $a\in\R^d$:
\begin{equation} \label{Add_proprime}
Q_t(g)_{a} = Q_t(g_{a})	
\end{equation}
where $g_{a}(x,y) := g(x,y+a)$ for every $(x,y)\in \X\times \R^d$. Let us introduce the following notation:
	\[ \widetilde{Q}_{s}(x;dx_1\times dy_1) := Q_{s}(x,0;dx_1\times dy_1).
\]
Then, we have: 
\begin{lem} \label{BasicFormula}
For any $\C$-valued function $g$ on $\X\times\R^d$ such that $\E[|g(X_u,Y_u)|]<\infty$ for every $u\in\T$, we have: 
\begin{equation} \label{Etat}
\E[ g(X_{s+t},Y_{s+t}) \mid \cF_s^{(X,Y)}]   = Q_t(g_{Y_s})(X_s,0) = \widetilde{Q}_{t}(g_{Y_s})(X_s) .
\end{equation}
or in terms of the increments of the process $(Y_t)_{t\in\T}$:
\begin{equation}  \label{Accroissements}
\E[ g(X_{s+t},Y_{s+t}-Y_s) \mid \cF_s^{(X,Y)}] = Q_t(g)(X_s,0)= \widetilde{Q}_{t}(g)(X_s)= \E_{(X_s,0)}[ g(X_{t},Y_{t})] .
\end{equation}
\end{lem}
\begin{proof}{} The two formula are derived as follows:
\begin{eqnarray*}
\E[ g(X_{s+t},Y_{s+t}) \mid \cF_s^{(X,Y)}]	& = & \E[ g(X_{s+t},Y_{s+t}) \mid X_s,Y_s] \qquad \text{ (Markov property)} \\
	 & = & Q_t (g)(X_s,Y_s)    \\
	 & = & Q_t(g_{Y_s})(X_s,0) \qquad \text{(from (\ref{Add_proprime}))}\\
	 & = & \widetilde{Q}_{t}(g_{Y_s})(X_s) ; \\
\E[ g(X_{s+t},Y_{s+t}-Y_s) \mid \cF_s^{(X,Y)}]
	& = & \E[ g(X_{s+t},Y_{s+t}-Y_s) \mid X_s,Y_s] \qquad \text{ (Markov property)} \\
	 & = & \E[ g_{-Y_s}(X_{s+t},Y_{s+t}) \mid X_s,Y_s] \\
	 & = & Q_t(g_{0})(X_s,0)=\widetilde{Q}_{t}(g)(X_s) \qquad \text{ (from (\ref{Etat})) }\\
	 &  = &  \E_{(X_s,0)}[ g(X_{t},Y_{t} )].
\end{eqnarray*}
\vspace*{-3mm}
\end{proof}
\begin{lem} \label{LoiConjointeAccroissements}
For every $n\ge 1$, any $\C$-valued function $g$ such that  for every $0\le u_1 \le \cdots \le u_n$ $$\E\big[| g(X_{u_1},Y_{u_1},X_{u_2},Y_{{u_2}}-Y_{u_{1}},\ldots,X_{u_n},Y_{{u_n}}-Y_{u_{n-1}})|\big]<\infty$$ 
we have  for any $s\ge 0$ and $t_1,\ldots, t_n\ge 0$:
\begin{eqnarray} \label{Eq_stationnaire}
& & \E[g(X_{s+t_1},Y_{s+t_1}-Y_{s},\ldots,X_{s+\sum_{i=1}^nt_i},Y_{s+\sum_{i=1}^n t_i}-Y_{s+\sum_{i=1}^{n-1} t_i}) \mid \cF_{s}^{(X,Y)}] \nonumber\\
& & = \int \widetilde{Q}_{s}(X_s;dx_1\times dz_1) \prod_{i=2}^{n} \widetilde{Q}_{s}(x_{i-1};dx_i\times dz_i) g(x_1,z_1,\ldots,x_n,z_n)  \nonumber\\
& &= (\mprod{1}{n}\widetilde{Q}_{t_{i}})(g)\big(X_{s}\big).
\end{eqnarray}
\end{lem}
\begin{proof}{}
 Lemma~\ref{BasicFormula} gives the case $n=1$. Let us check that Formula (\ref{Eq_stationnaire}) is valid for $n=2$. This can help the reader to follow the induction.
\begin{eqnarray*}
& & \E[g(X_{s+t_1},Y_{s+t_1}-Y_{s},X_{s+t_1+t_2},Y_{s+t_1+t_2}-Y_{s+t_1}) \mid \cF_{s}^{(X,Y)}]\\
& & =  \E\left[\E\big[g(X_{s+t_1},Y_{s+t_1}-Y_{s},X_{s+t_1+t_2},Y_{s+t_1+t_2}-Y_{s+t_1}) \mid \cF_{s+t_1}^{(X,Y)} \big] \mid \cF_{s}^{(X,Y)}\right]\\
& & =  \E\left[\E\big[g(X_{s+t_1},Y_{s+t_1}-Y_{s},X_{s+t_1+t_2},Y_{s+t_1+t_2}-Y_{s+t_1}) \mid X_{s+t_1},Y_{s+t_1} \big] \mid \cF_{s}^{(X,Y)}\right]\\
& & =  \E\big[ \int g(X_{s+t_1},Y_{s+t_1}-Y_{s},x_2,y_2-Y_{s+t_1}) Q_{t_2}(X_{s+t_1},Y_{s+t_1};dx_2\times dy_2) \mid \cF_{s}^{(X,Y)}\big]\\
& & =  \E\big[ \int g(X_{s+t_1},Y_{s+t_1}-Y_{s},x_2,z_2) \widetilde{Q}_{t_2}(X_{s+t_1};dx_2\times dz_2) \mid \cF_{s}^{(X,Y)}\big] \text{ (using (\ref{Add_pro})) }\\
& & =  \int \widetilde{Q}_{t_1}(X_{s};dx_1\times dz_1) \int \widetilde{Q}_{t_2}(x_1;dx_2\times dz_2)g(x_1,z_1,x_2,z_2)  \text{ (using (\ref{Accroissements}))}\\ \\
& & =(\widetilde{Q}_{t_{1}} \otimes \widetilde{Q}_{t_{2}})(g)\big(X_{s}\big).
\end{eqnarray*}

Let us now complete the induction. Assume that Property~(\ref{Eq_stationnaire}) is valid for $n-1$. Then
\begin{eqnarray*}
& & \E[g(X_{s+t_1},Y_{s+t_1}-Y_{s},\ldots,X_{s+\sum_{i=1}^n t_i},Y_{s+\sum_{i=1}^nt_i}-Y_{s+\sum_{i=1}^{n-1}t_i}) \mid \cF_{s}^{(X,Y)}]\\
& & =  \E\left[\E\big[g(X_{s+t_1},Y_{s+t_1}-Y_{s},\ldots,X_{s+\sum_{i=1}^nt_i},Y_{s+\sum_{i=1}^nt_i}-Y_{s+\sum_{i=1}^{n-1}t_i}) \mid \cF_{s+t_1}^{(X,Y)} \big] \mid \cF_{s}^{(X,Y)}\right]\\
&  & =\E\left[ (\mprod{2}{n}\widetilde{Q}_{t_{i}})\big(g(X_{s+t_1},Y_{s+t_1}-Y_s,\cdot,\cdots,\cdot)\big)\big(X_{s+t_1}\big) \mid \cF_{s}^{(X,Y)}\right] \text{ (induction)} \\
& & = \big(\widetilde{Q}_{t_{1}} \otimes ( \mprod{2}{n}\widetilde{Q}_{t_{i}}))(g)\big(X_{s}\big) \quad  \text{ (using (\ref{Accroissements})).}
\end{eqnarray*}
\end{proof}

\begin{cor} \label{Cor_statio} Under (AS\ref{AS1}),  
the following properties hold.
\leftmargini 1.3em
\begin{enumerate}
	\item The process $(Y_{t})_{t\in \T}$ has stationary increments, i.e.~
\begin{equation} \label{Acc_stationnaire}
 \E_{\pi,0}\big[g(Y_{s+t_1}-Y_{s},\ldots,Y_{s+\sum_{i=1}^nt_i}-Y_{s+\sum_{i=1}^{n-1}t_i})\big] = 
\pi\big((\otimes_{1}^{n}\widetilde{Q}_{t_{i}})(g)\big)
\end{equation}
does not depend on $s$ for any function $g$ as in Lemma~\ref{LoiConjointeAccroissements}. 
\item If $\E_{\pi,0}[|Y_{u}|] < \infty$ for every $u\in\T$, then: 
\[ \forall (s,t)\in\T^2,\ \E_{\pi,0}[Y_{s+t}] = \E_{\pi,0}[Y_t] + \E_{\pi,0}[Y_s].\] 
\item $(\xi_n:=Y_{n}-Y_{n-1})_{n\in\N^*}$ is a stationary sequence of $\R^d$-valued random variables and if $h$ is a $\C$-valued function such that $\E_{\pi,0}\big[|h(\xi_1,\ldots,\xi_n)|^2\big]=1$, then $\widetilde{Q}_1^{\otimes n} (h)\in\L^2(\pi)$ with
	\begin{equation} \label{L2Norme}
	\| \widetilde{Q}_1^{\otimes n} (h) \|_{2} \le 1,
\end{equation}
where $\widetilde{Q}_1^{\otimes n}$ denotes the $n$-fold kernel product $\mprod{1}{n}\widetilde{Q}_1$.
\end{enumerate}
\end{cor}
\begin{proof}{}
Take the expectation of (\ref{Eq_stationnaire})  with respect to the probability mesure $\pi$: 
\begin{eqnarray*}
\lefteqn{ \E_{\pi,0}\big[g(Y_{s+t_1}-Y_{s},\ldots,Y_{s+\sum_{i=1}^n t_i}-Y_{s+\sum_{i=1}^{n-1} t_i})\big] } \\
& = & 	\E_{\pi,0} \left[\big(\mprod{1}{n}\widetilde{Q}_{t_{i}}\big)(g)(X_{s}) \right] = 
\E_{\pi P_s,0} \left[\big(\mprod{1}{n}\widetilde{Q}_{t_{i}}\big)(g)(X_{0}) \right] \\
& = & \E_{\pi,0} \left[\big(\mprod{1}{n}\widetilde{Q}_{t_{i}}\big)(g)(X_{0}) \right] \quad \text{ (invariance property of $\pi$).}
\end{eqnarray*}
The second property is deduced from the stationarity of the increments of $(Y_{t})_{t\in \T}$. Indeed, we can write $\E_{\pi,0}[Y_t]= \E_{\pi,0}[Y_{s+t}-Y_s] = \E_{\pi,0}[Y_{s+t}] - \E_{\pi,0}[Y_s]$. 
That $(\xi_n)_{n\in\N^*}$ is a stationary sequence of random variables follows from (\ref{Acc_stationnaire}) with $s=0,t_1=\cdots=t_n=1$. The last property follows from (\ref{Eq_stationnaire}) and the Jensen inequality 
\begin{equation*}
\big\|\widetilde{Q}_{1}^{\otimes n} h\|_{2}^2
 =  \big\|\E_{(\cdot,0)}\big[h(\xi_{1},\ldots,\xi_{n})\big]\|_{2}^2 \le   \E_{\pi,0}\big[|h(\xi_{1},\ldots,\xi_{n})| ^2\big]=1. 
\end{equation*}
\end{proof}

\begin{lem} \label{Correlation} Let $\xi_n := Y_n-Y_{n-1}$ for $n\in\N^*$ (recall that $Y_0=0$ a.s.). Let $f$ and $h$ be two $\C$-valued functions such that $\E_{\pi,0}\big[| f(\xi_1,\ldots,\xi_n)| ^2\big]<\infty$ and  $\E_{\pi,0}\big[| h(\xi_{n+t},\ldots,\xi_{n+t+m})|^2\big]<\infty$ for  $(t,n,m)\in(\N^*)^3$. Under (AS\ref{AS1}), the covariance has the following form   
\begin{equation} \label{Cov}
 \mathrm{Cov}_{\pi,0}\big(f(\xi_1,\ldots,\xi_n); h(\xi_{n+t},\ldots,\xi_{n+t+m})\big) = \E_{\pi,0}\left[ f(\xi_1,\ldots,\xi_n) (P_{t-1}-\Pi) \big(\widetilde{Q}_{1}^{\otimes m+1}(h)\big)(X_n)\right]
\end{equation}
where $\widetilde{Q}_{1}^{\otimes m+1}$ denotes the $(m+1)$-fold kernel product   $\mprod{1}{m+1}\widetilde{Q}_{1}$.
\end{lem}
\begin{proof}{}
Apply formula~(\ref{Eq_stationnaire}) with $\E_{(x,0)}$ to the specific function $g(x_1,z_1,\ldots,x_{n+t+m},z_{n+t+m}) = f(z_1,\ldots,z_{n}) \newline \times h(z_{n+t},\ldots,z_{n+t+m})$ with $t,n,m\ge 1$:
\begin{eqnarray*}
& &\E_{(x,0)}\big[f(\xi_1,\ldots,\xi_n) h(\xi_{n+t},\ldots,\xi_{n+t+m})\big] = \widetilde{Q}_{1}^{\otimes n+m+t}(g)(x) \\
  &=&  \int_{(\X\times \R^d)^{n+t+m}} \widetilde{Q}_{1}(x;dx_1\times dz_1) \prod_{i=2}^{n+t+m} \widetilde{Q}_{1}(x_{i-1};dx_i\times dz_i) g(x_1,\ldots,z_{n+t+m})\\
  &=& \int_{(\X\times \R^d)^{n}} \widetilde{Q}_{1}(x;dx_1\times dz_1) \prod_{i=2}^{n} \widetilde{Q}_{1}(x_{i-1};dx_i\times dz_i)f(z_1,\ldots,z_{n})  \\
  & &   \quad \times \int_{(\X\times \R^d)^{t-1}}\prod_{i=n+1}^{n+t-1} \widetilde{Q}_{1}(x_{i-1};dx_i\times dz_i)  \\
  & & \quad \times \int_{(\X\times \R^d)^{m+1}} \prod_{i=n+t}^{n+t+m} \widetilde{Q}_{1}(x_{i-1};dx_i\times dz_i) h(z_{n+t},\ldots,z_{n+t+m}).
  \end{eqnarray*}  
  The second term reduces to $\int_{\X^{t-1}}\prod_{i=n+1}^{n+t-1} \widetilde{Q}_{1}(x_{i-1};dx_i\times \R^d)= \int_{\X^{t-1}}\prod_{i=n+1}^{n+t-1} P_{1}(x_{i-1};dx_i) = \int_{\X} P_{t-1}(x_n;dx_{n+t-1})$. The third is $\widetilde{Q}_{1}^{\otimes m+1}(h)(x_{n+t-1})$. Then we have 
\begin{eqnarray*}
& &\E_{(x,0)}\big[f(\xi_1,\ldots,\xi_n) h(\xi_{n+t},\ldots,\xi_{n+t+m})\big]\\
& = & \int_{(\X\times \R^d)^{n}} \widetilde{Q}_{1}(x;dx_1\times dz_1) \prod_{i=2}^{n} \widetilde{Q}_{1}(x_{i-1};dx_i\times dz_i) f(z_1,\ldots,z_{n})   P_{t-1} (\widetilde{Q}_{1}^{\otimes m+1}(h)\big)(x_n) \\
  &=& \E_{(x,0)}\left[ f(\xi_1,\ldots,\xi_n) P_{t-1} (\widetilde{Q}_{1}^{\otimes m+1}h\big)(X_n)\right] \text{ (using (\ref{Eq_stationnaire}) with $\E_{(x,0)}$)}. 
	\end{eqnarray*}
	Then, integrating against the probability measure $\pi(\cdot)$ gives 
\begin{equation} \label{inter}
\E_{\pi,0}\big[f(\xi_1,\ldots,\xi_n) h(\xi_{n+t},\ldots,\xi_{n+t+m})\big] =
\E_{\pi,0}\left[ f(\xi_1,\ldots,\xi_n) P_{t-1} (\widetilde{Q}_{1}^{\otimes m+1}h\big)(X_n)\right].
  \end{equation}
Since $\Pi \big(\widetilde{Q}_{1}^{\otimes m+1}(h)\big)(x) = \pi \big(\widetilde{Q}_{1}^{\otimes m+1}(h)\big)$ for every $x\in \X$, we obtain 
	\begin{eqnarray*}
	\E_{\pi,0}\left[ f(\xi_1,\ldots,\xi_n) \Pi (\widetilde{Q}_{1}^{\otimes m+1}(h)\big)(X_n)\right]  
	& = & \E_{\pi,0}[f(\xi_1,\ldots,\xi_n)] \pi \big(\widetilde{Q}_{1}^{\otimes m+1}(h)\big) \\
	& = & \E_{\pi,0}[f(\xi_1,\ldots,\xi_n)] \E_{\pi,0}[h(\xi_{n+t},\ldots,\xi_{n+t+m})] 
\end{eqnarray*}
where the last equality follows from (\ref{Acc_stationnaire}). 
\end{proof}
\begin{rem}  
We can prove a continuous-time counterpart of Lemma~\ref{Correlation}. But, we restrict ourself to the discrete-time version because this is the version we need in the paper. 
\end{rem}

\subsection{$\rho$-mixing property of $(Y_{n}-Y_{n-1})_{n\ge 1}$ for discrete-time stationary MAPs}

Let us recall some basic facts on the $\rho$-mixing of a (strictly) stationary sequence of random variables $(\xi_n)_{n\in\N^*}$ (e.g.~see \cite{Bra05-1}). For any $p\in \N^*$ and $q\in \N^*\cup\{\infty\}$ with $p\le q$, $\cG^{q}_{p}:= \sigma(\xi_p,\ldots,\xi_q)$  denotes the $\sigma$-algebra generated by random variables $\xi_p,\ldots,\xi_q$.  
The $\rho$-mixing coefficient at horizon $t>0$, denoted by $\rho(t)$, is defined by  
\begin{equation} \label{Def_rhot}
	\rho(t) := \sup_{n\in\N^*}\sup\left\{ |\mathrm{Corr}(f;h)| \ f\in\L^2(\cG^{n}_{1}), h\in \L^2(\cG_{n+t}^{\infty}) \right\}.
\end{equation}
where $\mathrm{Corr}(f;h)$ is the correlation coefficient of the two random variables $f$ and $g$.
In fact, $\rho$-mixing coefficient may be computed as follows from \cite[Prop~3.18]{Bra05-1}: for $t>0$
\begin{eqnarray} \label{Calcul_rho}
	\rho(t) & = & \sup_{n\in\N^*} 
\sup_{m\in\N^*} \sup\left\{ |\mathrm{Corr}(f;h)| \ f\in\L^2(\cG^{n}_{1}), h\in \L^2(\cG_{n+t}^{n+t+m})\right\}  .
\end{eqnarray}
The  stationary sequence $(\xi_n)_{n\in\N^*}$ is said to be \textit{$\rho$-mixing} if 
\begin{equation*}
	\lim_{t\r +\infty} \rho(t) = 0.
\end{equation*}

We know from condition (AS\ref{AS-spectralgap}) that $(X_n)_{n\in\N}$ is $\rho$-mixing \cite{Ros71}. In the special case when $Y_{n}:=\sum_{k=1}^n\xi(X_k)$, it is clear that $(Y_{n}-Y_{n-1})_{n\ge 1}$ is also $\rho$-mixing from \cite[p. 28]{Bra05-1}. We extend this fact to general $(Y_n)_{n\in\N}$ in the next proposition.
\begin{pro}[$\rho$-mixing] \label{mixing}
Under (AS\ref{AS1}-AS\ref{AS-spectralgap}), the stationary sequence $(\xi_n:=Y_{n}-Y_{n-1})_{n \in \N^*}$ is $\rho$-mixing at an exponential rate: there exists $\varepsilon >0$ such that 
\begin{equation*}
	\rho(t) = O\big(\exp(-\varepsilon t)\big).
\end{equation*}
\end{pro}
\begin{proof}{} For the sake of simplicity, assume that $d=1$. First, note that the random variables $f$ and $h$ in (\ref{Calcul_rho}) can be assumed to be of $\L^2$-norm 1. Thus, we just have to deal with covariances. Recall that $(\xi_n:=Y_{n}-Y_{n-1})_{n\in\N^*}$ is known to be stationary under (AS\ref{AS1}) from Corollary~\ref{Cor_statio}. The $\sigma$-algebra $\cG^{n}_{1}$ and $\cG_{n+t}^{n+t+m}$ in the mixing coefficients will be relative to the stationary sequence  $(\xi_n)_{n\in\N^*}$.
Second, let us consider two $\L^2$-normed random variables $f(\xi_1,\ldots,\xi_n)\in\L^2(\cG^{n}_{1})$, $h( \xi_{n+t}, \ldots, \xi_{n+t+m})\in  \L^2(\cG_{n+t}^{n+t+m})$. For any $m\ge 1$,  
the map $x \mapsto  \widetilde{Q}_{1}^{\otimes m+1} (h)(x)$ belongs to $\L^2(\pi)$ and we have $\big\|\widetilde{Q}_{1}^{\otimes m+1} (h)\|_{2} \le 1$ from  Corollary~\ref{Cor_statio} (see (\ref{L2Norme})). 
Since $P_{t}$ and $\Pi$ are contractions on $\L^2(\pi)$, we have $(P_{t-1}-\Pi)\big(  \widetilde{Q}_{1}^{\otimes m+1} (h)\big) \in \L^2(\pi)$. 

The Cauchy-Schwarz inequality and the last comments allow us to write from (\ref{Cov})
\begin{eqnarray*}
\mathrm{Cov}(f;h)^2	 & \le & 
\E_{\pi,0}\left[ |f(\xi_1,\ldots,\xi_n)|^2\right] \E_{\pi,0}\left[ \big|(P_{t-1}-\Pi)\big( \widetilde{Q}_{1}^{\otimes m+1}(h)\big)(X_{n})\big| ^2\, \right] \nonumber \\
& & =  \E_{\pi,0}\left[ \big|(P_{t-1}-\Pi)\big( {\widetilde{Q}_{1}}^{\otimes m+1} (h)\big)(X_{0})\big| ^2\, \right] \quad \text{ ($\pi$ is $P_n$-invariant)}\\
&& =  \big\|(P_{t-1}-\Pi) \big( \widetilde{Q}_{1}^{\otimes m+1} (h)\big)\big\|_{2}^2 \\
& \le & \| P_{t-1} - \Pi \|_{2}^2\, \big\| \widetilde{Q}_{1}^{\otimes m+1} (h)\big\|_{2}^2 \quad \text{ (since $\widetilde{Q}_{1}^{\otimes m+1} (h) \in \L^2(\pi)$)}\\
& \le &\| P_{t-1} - \Pi \|_{2}^2 \qquad \text{ (since $\big\|\widetilde{Q}_{1}^{\otimes m+1} (h)\|_{2} \le 1$)}. 
\end{eqnarray*}
Therefore, it follows that for every $t \ge 1$:
\begin{eqnarray*} 
 \sup\left\{ |\mathrm{Corr}(f;h)| \ f\in\L^2(\cG^{n}_{1}), h\in \L^2(\cG_{n+t}^{n+t+m})\right\} & \le & 
 \| P_{t-1} - \Pi \|_{2}  .
\end{eqnarray*}
The right hand side term in the inequality above does not depend on $m$ and $n$, so that we obtain from (\ref{Calcul_rho})
\begin{equation*}
\rho (t)  \le \| P_{t-1} - \Pi \|_{2}.
\end{equation*}
The proof is completed by using the exponential estimate (\ref{Expo}) of $\| P_{t-1} - \Pi \|_{2}$ under (AS\ref{AS-spectralgap}).  
\end{proof}

\subsection{Central limit theorem for the Markov additive processes} \label{clt-discrete}

In a first step, we consider a discrete-time $\X\times\R^d$-valued MAP, $(X_n,Y_n)_{n\in \N}$, for which the driving Markov chain $(X_n)_{n\in \N}$ is assumed to satisfy  (AS\ref{AS1}-AS\ref{AS-spectralgap}). Recall that Condition (\ref{AS3d}) for $\alpha =2$ is 
\begin{equation*} 
	 \E_{\pi,0}\big[|Y_1|^{2}\big] < \infty. 
\end{equation*}
 This condition implies that $\E_{\pi,0}[|Y_1|] < \infty$, and we suppose that $\E_{\pi,0}[Y_1] = 0$ for convenience (if not, replace $Y_n$ by $Y_n -\E_{\pi,0}[Y_n]=Y_n - n\E_{\pi,0}[Y_1]$ from Corollary~\ref{Cor_statio}). 

We know from Proposition~\ref{mixing} that $(Y_{n+1}-Y_n)_{n\in\N}$ is stationary and is exponentially $\rho$-mixing when (AS\ref{AS1}-AS\ref{AS-spectralgap}) hold.
Then, under the expected moment condition $\E_{\pi,0}[|Y_1|^2]<\infty$, the CLT for $(Y_n)_{n\in\N^*}$ follows from \cite{Ibra75,Pel87} (e.g.~see \cite[Th~11.4]{Bra05-1}). 
 To the best of our knowledge, Theorem~\ref{clt-discret-multi} for general MAP is new. 
The notation $\cN(0,0)$ stands for the Dirac distribution at 0. 
\begin{thm} \label{clt-discret-multi} 
Under (AS\ref{AS1}-AS\ref{AS-spectralgap}) and (\ref{AS3d}) for $\alpha= 2$, $(Y_n/\sqrt n)_{n\in\N}$ converges in  distribution when $n\r+\infty$ to the $d$-dimensional Gaussian law  $\cN(0,\Sigma)$, where $\Sigma$ is the asymptotic covariance $d\times d$-matrix 
$$\Sigma \ : =\, \lim_n\frac{1}{n}\, \E_{\pi,0}\left[\, Y_n \, Y_n^*\, \right]$$
where the symbol $*$ denotes the transpose operator.   
\end{thm}

That $(Y_n/\sqrt n)_{n\in\N}$ satisfies the CLT under the  condition $\E_{\pi,0}[|Y_1|^{2}] < \infty$ was known in some  cases.  Such standard situations are recalled in the two next remarks (with $d=1$ to simplify). 
\begin{rem}[Martingale method] \label{clt-discret-xi} 
If there exists a measurable function $\xi : \X\r \R$  such that $Y_n-Y_{n-1} = \xi(X_n)$ and $\E_{\pi,0}[|Y_1|^{2}]=\pi(\xi^2)<\infty$, then $(Y_n/\sqrt n)_{n\in\N}$ converges in  distribution to the Gaussian law  $\cN(0,\sigma^2)$ where $\sigma^2 =  \pi(\xi^2) + 2\sum_{\ell=1}^{+\infty}\pi(\xi P^\ell\xi) \in[0,+\infty)$. 
This result follows from the Gordin-Lifsic theorem \cite{Gor78}. Indeed, (AS\ref{AS-spectralgap}) implies that $(X_n)_{n\in\N}$ is ergodic and that there is a solution $\tilde \xi\in\L^2(\pi)$ to the Poisson equation: $\tilde \xi - P\tilde \xi = \xi$. Then,  the difference martingale method of \cite{Gor78} can be used to obtain the CLT. 
\end{rem}

\begin{rem}[Uniform ergodicity] \label{clt-discret-doeblin} 
Recall that the Markov chain $(X_n)_{n\in \N}$ is said to be uniformly ergodic if $\lim_{n\r+\infty}\|P^n-\Pi\|_\infty=0$. This property implies (AS\ref{AS-spectralgap}) (but is stronger) and is fulfilled if and only if $(X_n)_{n\in \N}$ is ergodic, aperiodic and satisfies the Doeblin condition ($D_0$). In addition, for an aperiodic and ergodic Markov chain $(X_n)_{n\in\N}$, Doeblin's condition is equivalent to the uniform mixing (or $\phi$-mixing) property, and then, the $\phi$-mixing coefficients go to 0 at least exponentially fast (see \cite{Ros71,Bra05}). 

Set $\xi_n := Y_n-Y_{n-1}$. If $(X_n)_{n\in \N}$ is uniformly ergodic and if  $\E_{\pi,0}[Y_1^2]<\infty$, then the real number $\sigma^2 =  \E_{\pi,0}[\xi_1^2] + 2 \sum_{\ell=1}^{+\infty}\E_{\pi,0}[\xi_1\, \xi_\ell]$ is well-defined in $[0,+\infty)$. If $\sigma^2 > 0$, then the sequence $(Y_n/\sqrt n)_{n\in\N}$ converges in  distribution to $\cN(0,\sigma^2)$ \cite{GriOpr76}. This  CLT is established  as follows:  the stationarity and the uniform  ergodicity of $(X_n)_{n\in \N}$ extend to the sequence $(\xi_n)_{n\in \N}$, and  the $\phi$-mixing coefficients of $(\xi_n)_{n\in \N}$ also go to 0 at an exponential rate (see \cite[Rk.~4, Lem.~1]{GriOpr76}). The proof is completed using \cite[Th.18.5.2]{IbrLin71}. 
\end{rem}

The CLT for a continuous-time MAP $(X_t,Y_t)_{t \ge 0}$ is deduced from the discrete-time statement. 
\begin{thm} \label{CLT_tpscont}Under (AS\ref{AS1}-AS\ref{AS-spectralgap}) and (AS\ref{AS-alpha}) with $\alpha= 2$,  $(Y_t/\sqrt{t})_{t>0}$ satisfies a CLT.
\end{thm}
\begin{proof}{}
 For $t\in[0,+\infty)$, we set $t=n+v$, where $n$ is the integer part of $t$ and $v\in[0,1)$. We can write:
\begin{equation} \label{decompo}
	\frac{Y_t}{\sqrt{t}} = \frac{(Y_t - Y_{n})}{\sqrt{t}} + \frac{\sqrt{n}}{\sqrt{t}} \frac{Y_{n}}{\sqrt{n}} . 
\end{equation}
Recall that $(Q_t)_{t\ge 0}$ is the transition semigroup of $(X_t,Y_t)_{t \ge 0}$. It is easily checked that the MAP $(X_t,Y_t)_{t \ge 0}$ ``sampled'' at discrete instants, $(X_n,Y_n)_{n \in \N}$, is a discrete-time MAP with transition kernel $Q:= Q_1$ which satisfies (AS\ref{AS1}-AS\ref{AS-spectralgap}) and (\ref{AS3d}). 
 Therefore,   $(Y_{n}/\sqrt{n})_{n\in\N}$ satisfies a CLT thanks to Theorem~\ref{clt-discret-multi}. Finally, the sequence $((Y_t - Y_{n})/\sqrt{t})_{t >0}$  converges in probability to 0 from the Tchebychev inequality and condition (\ref{AS3d}): 
\begin{eqnarray*}
	\P_{\pi,0}\big\{|Y_t - Y_{n}| > \sqrt{t} \varepsilon \big\} & = & \P_{\pi,0}\big\{|Y_{v}| > \sqrt{t} \varepsilon \big\} \quad \text{(stationary increments)} \\
	& \le &  \frac{\E_{\pi,0}\big[| Y_{v}|^2\big] }{t\varepsilon^2}  \le  \frac{\sup_{v\in(0,1]}\E_{\pi,0}\big[|Y_v|^2\big] }{t\varepsilon^2} \xrightarrow[t\r +\infty]{} 0.
\end{eqnarray*}
Therefore, $(Y_t/\sqrt{t})_{t>0}$ satisfies a CLT from (\ref{decompo}). 
\end{proof}

\begin{rem}[FCLT] 
Proposition~\ref{mixing} allows us to deduce from \cite[Th. 19.2]{Bil95} that a functional central limit theorem also holds ($d=1$). That is, under the assumptions of Theorem~\ref{clt-discret-multi}, we have: 
	\begin{equation} \label{FCLT}
	 \left(\frac{Y_{\lfloor n t \rfloor}}{\sqrt{n}}\right)_{t\ge 0} \xrightarrow[n\r +\infty]{\cL} (B_t)_{t\ge 0}
\end{equation}
as random elements of $D[0,\infty)$, the Skorokhod space of cadlag functions on $\R_+$, and where $(B_t)_{t\ge 0}$ is a Brownian motion with zero drift and some variance parameter. Let us give some comments on the FCLT relevant to our context.
\begin{enumerate}[(a)]
\item  The case of a discrete-time MAP $(X_n,Y_n)_{n\in\N}$ with $(X_n)_{n\in\N}$ satisfying the Doeblin condition  is covered by \cite{GriOpr76}.  $(Y_n)_{n\in\N}$ is shown to be $\phi$-mixing and a FCLT 
for $\phi$-mixing sequences is used.  
We extend their approach to our case of $\L^2$-spectral gap.

	\item Under (AS\ref{AS1}-AS\ref{AS-spectralgap}) and the expected moment condition of order 2, Maigret \cite{Mai78} has established a FCLT for $(Y_n :=  
\xi(X_{n-1},X_n))_{n\in\N^*}$ in the specific case where $(X_n)_{n\in\N}$ is Harris-recurrent. It is worth noticing that Condition~(AS\ref{AS-spectralgap}) cannot be compared with the Harris-recurrence property.

\item If $(X_t)_{t\ge 0}$ is a stationary ergodic Markov process with a strongly continuous transition semigroup $(P_t)_{t\ge 0}$ on $\L^2(\pi)$, the following convergence holds for any $f\in\L^2(\pi)$ such that $\pi(f)=0$ \cite[Th. 2.1, Prop. 2.3]{Bha82} (see also \cite{Tou83} in the Harris-recurrent case): 
	\[ \left(\frac{1}{\sqrt{n}} \int_0^{nt}\xi(X_s) ds \right)_{t\ge 0} \xrightarrow[n\r +\infty]{\cL} (B_t)_{t\ge 0}
\]
where $(B_t)_{t\ge 0}$ is a Brownian motion with zero drift and some variance parameter.
Set $Y_t :=  \int_0^{t}\xi(X_s) ds$. Since $\xi\in\L^2(\pi)$, we have $\E_{\pi,0}\big[|Y_t|^{2}\big]\le \pi(|\xi|^2)$ for every $t\in (0,1]$, thus (AS\ref{AS-alpha}) with $\alpha=2$ is true.  
Then, the convergence result above is easily deduced from (\ref{FCLT}) using the discrete-time stationary MAP $(X_n,Y_n)_{n\ge 1}$ introduced in the proof of Theorem~\ref{CLT_tpscont}. 

\item Glynn and Whitt  deal with the integral functional of a regenerative process in \cite{GlyWhi93,GlyWhi02}. Their results apply to a Markov process which is a specific instance of a  regenerative process. Conditions for the CLT (FCLT) to hold are expressed in terms of a second moment on the increments $Y_{T_1}:=\int_0^{T_1} \xi(X_s) ds$ of the process $(Y_t)_{t> 0}$ over a regeneration cycle of length $T_1$ (and an additional condition of negligeability in probability of $\sup_{0\le s \le T_1}|Y_{s}|$). The fact that we only consider the Markov case makes our conditions easier to check.

\end{enumerate}
\end{rem}

\section{Refinements of the central limit theorem for MAPs} \label{Section_Rafinements}

Let $(X_t,Y_t)_{t\in \T}$ be a MAP taking values in $\X\times\R^d$. The canonical scalar product on $\R^d$ is denoted by $\langle\cdot,\cdot\rangle$. The Fourier operators associated with $(X_t,Y_t)_{t\in \T}$ are introduced in the next subsection and  are shown to satisfy a semigroup property. In the discrete-time case, precise expansions of the characteristic function of the additive component $Y_t$ can be deduced from \cite{HerPen09} under (AS\ref{AS-spectralgap}). They are central to the derivation of our limit theorems in this section. Limit theorems are  first considered for discrete-time MAPs. A local limit theorem, a Berry-Esseen bound and a first-order Edgeworth expansion are obtained.
The continuous-time case is addressed thanks to the basic reduction to the discrete-time case used for the CLT.

\subsection{Fourier operators. A semigroup property} \label{SS_Fourier} 

For any $t\in\T$ and $\zeta\in\R^d$, we consider the linear operator $S_t(\zeta)$ acting (in a first step) on the space of bounded measurable functions $f : \X\r\C$ as follows:  
\begin{equation} \label{Fourier}
	\forall x\in \X,\ \ \ \big(S_t(\zeta)f\big)(x) := \E_{(x,0)}\big[e^{i \langle \zeta , Y_t \rangle}\, f(X_t)\big]. 
\end{equation}
Note that $S_t(0) = P_t$. In the discrete-time case, $S_1(\zeta)$ corresponds to the Fourier operator which was first introduced by Nagaev \cite{Nag57} in the special case when $Y_n = \sum_{k=1}^n\xi(X_k)$ (see \cite{GuiHar88,HenHer01} and the reference therein), and was  extended  to discrete-time MAPs in \cite{guivarch,bab} to prove local limit theorems and renewals theorems (see also \cite{FuhLai01}). All these works are based on the following formula (see Proposition~\ref{pro-semi-grpe} below): 
\begin{equation} \label{semi-gpe-discret}
\forall \zeta\in\R^d, \forall n\in\N,\qquad \big(S_n(\zeta)f\big)(x) := \E_{(x,0)}\big[e^{i \langle \zeta , Y_n \rangle}\, f(X_n)\big] = \big(S_1(\zeta)^n f\big)(x).
\end{equation}
This formula clearly reads as the semigroup property: $S_{m+n}(\zeta) = S_m(\zeta) \circ S_n(\zeta)$. In the continuous-time, it seems that the operators $S_t(\zeta)$ were first introduced in \cite{HitShi70} for investigating AFs of continuous-time Markov processes on a compact metric state space $\X$. In \cite{HitShi70}, $P_t$ was assumed to have a spectral gap on the space of all  continuous $\C$-valued functions on $\X$, and $(S_t(\zeta))_{t>0}$ was thought of as a semigroup (see (\ref{semi-group}) below) on this space.  

Here, in view of (AS\ref{AS-spectralgap}),  the above mentioned semigroup property has to be considered on the Lebesgue spaces $\L^p(\pi)$ ($1\le p \le \infty$). 
\begin{pro} \label{pro-semi-grpe} 
For all $t\in\T$ and $\zeta\in\R^d$, $S_t(\zeta)$ defines a linear contraction on $\L^p(\pi)$, and we have:  
\begin{equation} \label{semi-group}
\forall \zeta\in\R^d, \quad \forall(s,t)\in\T^2,\ \ S_{t+s}(\zeta) = S_t(\zeta) \circ S_s(\zeta) \tag{SG}.
\end{equation}
In particular, Relation~(\ref{semi-gpe-discret}) holds for all $f\in\L^p(\pi)$. 
\end{pro}
\begin{proof}{}
 The first assertion is easy to prove.  Next, for any $\zeta\in\R^d$ and $f\in\L^p(\pi)$, let us set: $g(x,y) := f(x)\, e^{i \langle \zeta , y \rangle }$ with $x\in \X$ and $y\in\R^d$. Then, using the Markov property and Lemma~\ref{BasicFormula}:	
 \begin{eqnarray*}
\lefteqn{(S_{t+s}(\zeta)f)(x) := \E_{(x,0)}\big[ e^{i \langle \zeta, Y_{t+s}\rangle}f(X_{t+s}) \big]}\\
& = & \E_{(x,0)}\big[ \E_{(x,0)}[ e^{i \langle \zeta, Y_{t+s}\rangle}f(X_{t+s})\mid \cF_s^{(X,Y)} ] \big]= 
\E_{(x,0)}\big[ (Q_t g_{Y_s})(X_s,0) \big] \\
& = & \E_{(x,0)}[ e^{i\langle \zeta,Y_s\rangle} (Q_t g)(X_s,0)] = \E_{(x,0)}\left[  e^{i\langle \zeta,Y_s\rangle})  \E_{X_s,0}[ f(X_{t}) e^{i\langle \zeta,Y_t\rangle} ]\right] \\
& = & \E_{(x,0)}\big[  e^{i\langle \zeta,Y_s\rangle}  (S_{t}(\zeta)f)(X_s)\big] = \left(S_{s}(\zeta)\big(S_{t}(\zeta)f\big)\right)(x)
\end{eqnarray*}
the third equality results from: $g_{Y_s}(x,y) = f(x)\, e^{i \langle \zeta, (y+Y_s) \rangle } =  e^{i \langle \zeta , Y_s \rangle }\, g(x,y)$.
This gives the semigroup property (\ref{semi-group}). The last assertion is obvious. 
\end{proof}

\subsection{Expansions of the characteristic function of the additive component} \label{SS_expansions}

Here we assume that $(X_n,Y_n)_{n\in \N}$ is a discrete-time MAP taking values in $\X\times\R^d$ (possibly derived from a continuous-time MAP) such that the driving Markov chain $(X_n)_{n\in \N}$ is stationary and satisfies (AS\ref{AS-spectralgap}). This last property ensures that $S_1(0)$ has good spectral properties, and the iterates $S_1(\zeta)^n$ occurring in (\ref{semi-gpe-discret}) are studied using the Nagaev-Guivarc'h spectral method which consists in applying the perturbation theory to the Fourier operators $S_1(\zeta)$ for small $\zeta$. However using the standard perturbation theorem requires strong assumptions on $Y_1$. Here we shall appeal to the weak spectral method introduced in \cite{HenHer04} and based on the Keller-Liverani perturbation theorem \cite{KelLiv99}. This method is fully developed in the Markov framework in \cite[see references therein]{HerPen09}. 
In the sequel, $F^{(\ell)}$ denotes the partial derivative of order $\ell$ of a $\C$-valued function $F$ defined on an open subset of $\R^d$.

Conditions (AS\ref{AS1}-AS\ref{AS-spectralgap}) are assumed to hold throughout the subsection. 

\begin{pro} \label{carac-en-0-disc}
Let $m_0\in\N^*$. Under condition (\ref{AS3d}) for some 
$\alpha>m_0$, there exists a bounded open neighborhood ${\cal O}$ of $\zeta=0$ in $\R^d$ such that we have for all $f\in\L^{s}(\pi)$ with any $s>\frac{\alpha}{\alpha-m_0}$:  
\begin{equation}\label{egalite_fcara_casdiscret}
\forall n\in\N,\ \forall \zeta\in{\cal O},\ \ \ \E_{\pi,0}\big[e^{i\, \langle \zeta , Y_n \rangle}\, f(X_n)\big] = \lambda(\zeta)^n\, L(\zeta,f) + R_n(\zeta,f),
\end{equation}
where $\lambda(\cdot)$, $L(\cdot,f)$, $R_n(\cdot,f)$ are $\C$-valued functions of class $\cC^{m_0}$ on ${\cal O}$, with $\lambda(0)=1$ and $L(0,f) = \pi(f)$. Moreover, we have the following properties for $\ell=0,\ldots,m_0$: 
\begin{subequations}
\begin{eqnarray}
&- & \sup_{\zeta\in{\cal O}} |L^{(\ell)}(\zeta,f)|<\infty \label{L-discret} \\
&- & \exists \kappa\in(0,1),\ \ \sup_{\zeta\in{\cal O}} |R_n^{(\ell)}(\zeta,f)| = O(\kappa^n). \label{R-discret} 
\end{eqnarray}
\end{subequations}
If $f:=1_{\X}$, we have $R_n(0,1_{\X}) = 0$. 
\end{pro}
When $Y_n = \sum_{k=1}^n\xi(X_k)$, the above properties are proved in \cite[Sect.~7.3]{HerPen09} by using (\ref{semi-gpe-discret}) and some operator-type derivation arguments. For a general additive component $Y_n$, the method is the same\footnote{See the beginning of the appendix. In particular, mention that $\lambda(\zeta)$ is the dominant eigenvalue of $S_1(\zeta)$, $L(\zeta,\cdot)$ is related to the associated eigenprojection, and $\kappa$ can be chosen as $\kappa =(e^{-\varepsilon}+1)/2$ where $\varepsilon>0$ is defined in (\ref{Expo}). } using Lemmas~\ref{cont-noyau-discrete} and \ref{deri-noyau-discrete} below which slightly extend \cite[Lem.~4.2,7.4]{HerPen09}. 
Mention that, by using the same lemmas, Proposition~\ref{carac-en-0-disc} can also be  deduced from  \cite{Gou08} which  
specifies the method introduced in \cite{HenHer04,GouLiv06} to prove Taylor expansions of $\lambda(\cdot)$, $L(\cdot,f)$, $R_n(\cdot,f)$
\footnote{
As observed in  \cite{Gou08}, the passage from the Taylor expansions to the differentiability properties can be derived from \cite{Cam64}.
}.

The operator norm in the space $\cL(\L^p,\L^{p'})$ of the linear bounded operators from $\L^p(\pi)$ to $\L^{p'}(\pi)$ is denoted by $\|\cdot\|_{p,p'}$.
\begin{lem} \label{cont-noyau-discrete}
If $1\leq p' < p$, then the map $\zeta\mapsto S_1(\zeta)$ is  continuous from $\R^d$ to $\cL(\L^{p},\L^{p'})$.
\end{lem} 
\begin{proof}{}
 We have for $\zeta\in \R^d$, $\zeta_0\in \R^d$ and $f\in \L^p(\pi)$, thanks to H\"older's inequality
\begin{eqnarray*}
\big|(S_1(\zeta)-S_1(\zeta_0)) f(x)\big|^{p'}
&=&\left|\E_{(x,0)}\big[e^{i\, \langle \zeta , Y_1 \rangle}\, f(X_1)\big]-\E_{(x,0)}\big[e^{i\, \langle \zeta_0 , Y_1 \rangle}\, f(X_1)\big]\right|^{p'} \\
&\leq &\E_{(x,0)}\left[\big| e^{i\, \langle \zeta-\zeta_0 , Y_1 \rangle}-1\big|^{p'}|f(X_1)|^{p'}\right] \\
&\leq &2^{p'} \E_{(x,0)}\left[\min\big\{1,| \langle \zeta-\zeta_0 , Y_1 \rangle|\big\}^{p'}\, |f(X_1)|^{p'} \right] ,
\end{eqnarray*}
the last inequality resulting from the classic inequality $|e^{ia}-1|\leq 2 \min\big\{1,|a|\big\}$. An integration with respect to $\pi$ and the use of H\"older's inequality give 
\begin{equation*}
\begin{split}
\pi\big(|(S_1(\zeta)-S_1(\zeta_0))  f|^{p'}\big)
&\leq 2^{p'}  \E_{\pi,0}\big[\min\big\{1,| \langle \zeta-\zeta_0 , Y_1 \rangle|\big\}^{(pp')/(p-p')}\big]^{(p-p')/p} \,
\E_{\pi,0}\big[ |f(X_1)|^{p} \big]^{p'/p} \\
&\leq 2^{p'}  \left\|\min\big\{1,| \langle \zeta-\zeta_0 , Y_1 \rangle|\big\}\right\|_{(pp')/(p-p')}^{p'} \,
\|f\|_{p}^{p'} ,
\end{split}
\end{equation*}
since $\pi$ is invariant. Thus, we deduce  that  
$\|S_1(\zeta)-S_1(\zeta_0) \|_{p,p'} \leq 2 \big\|\min\big\{1,| \langle \zeta-\zeta_0 , Y_1 \rangle|\big\}\big\|_{(pp')/(p-p')}$ goes to $0$ when $|\zeta-\zeta_0| \r 0$ from Lebesgue's theorem. 
\end{proof}
\begin{lem} \label{deri-noyau-discrete} 
Assume that (\ref{AS3d}) holds for some $\alpha>m_0$ ($m_0\in\N^*$), and let $1\leq j \leq m_0$. 
If $p > 1$ and $p_j :=\alpha p/(\alpha+jp) \ge 1$, then $\zeta\mapsto S_1(\zeta)$ is $j$-times continuously differentiable from $\R^d$ to $\cL(\L^{p},\L^{p_j})$, and  $\sup_{\zeta\in\R^d}\|S_1^{(j)}(\zeta)\|_{p,p_j} \leq \E_{\pi,0}[|Y_1|^{\alpha}]^{j/\alpha}$. 
\end{lem} 
\begin{proof}{} 
For the sake of simplicity, we suppose that $d=1$. Below we consider any $\zeta\in \R$, $\zeta_0\in \R$ and $f\in \L^p(\pi)$. For $1\leq j \leq m_0$, define (formally) the following linear operator: 
\begin{equation*} 
	\forall x\in \X,\ \ \ \big(S_1^{(j)}(\zeta)f\big)(x) := \E_{(x,0)}\big[(i Y_1)^j \, e^{i \zeta  Y_1} \, f(X_1)\big]. 
\end{equation*}
First we have: 
\[
|S_1^{(j)}(\zeta) f(x)|^{p_j}
\leq  \E_{(x,0)}\left[|  Y_1|^{jp_j}\, |f(X_1)|^{p_j} \right] \]
so that, from H\"older's inequality,  
\[ \|S_1^{(j)}(\zeta) \|_{p,p_j}
\leq  \E_{\pi,0}\big[|  Y_1|^\alpha \big]^{j/\alpha}  .
\]
Second, define $\Delta:=S_1^{(j-1)}(\zeta)-S_1^{(j-1)}(\zeta_0)- (\zeta-\zeta_0) \ S_1^{(j)}(\zeta_0)$. Then, for $j\in\{1,\ldots,m_0-1\}$, we have thanks to the classic inequality $|e^{ia}-1-ia|\leq 2 |a|\min\big\{1,|a|\big\}$
\[
|\Delta f(x)|^{p_{j}}
\leq 2^{p_{j}} | \zeta-\zeta_0 |^{p_{j}} \E_{(x,0)}\left[\min\big\{1,| (\zeta-\zeta_0 ) Y_1|\big\}^{p_{j}}\, |Y_1|^{jp_{j}} \, |f(X_1)|^{p_{j}} \right]. 
\]
It follows from H\"older's inequality that the operator norm satisfies 
\[ \|\Delta \|_{p,p_{j}}
\leq 2 | \zeta-\zeta_0|\ \left\| \min\big\{1,| (\zeta-\zeta_0 ) Y_1|\big\} \ |Y_1|^{j}\right\|_{\alpha/j} .
\]
This proves that $S_1^{(j-1)}(\cdot)$ is differentiable from $\R$ to $\cL(\L^{p},\L^{p_j})$, and that its derivatives is $S_1^{(j)}$. 
Finally, we obtain:
\[
\big|(S_1^{(j)}(\zeta)-S_1^{(j)}(\zeta_0)) f(x)\big|^{p_j}
\leq 2^{p_j}  \E_{(x,0)}\left[\min\big\{1,| (\zeta-\zeta_0 ) Y_1|\big\}^{ p_j}|Y_1|^{j p_j}\, |f(X_1)|^{p_j} \right] 
\]
from which we deduce  that the operator norm  satisfies
\[ 
\|S_1^{(j)}(\zeta)-S_1^{(j)}(\zeta_0)\|_{p,p_j}
\leq 2  \left\|\min\big\{1,| (\zeta-\zeta_0 ) Y_1|\big\} \ |Y_1|^j\right\|_{\alpha/j} .
\]
Thus $S_1(\cdot)$ is $j$-times continuously differentiable from $\R$ to $\cL(\L^{p},\L^{p_j})$. 

\end{proof}

Next, let us return to our probabilistic context. Let $\nabla$ and $\mathrm{Hess}$ denote the gradient and the Hessian operators  respectively. In the following proposition,  the $d$-dimensional vector $\nabla\lambda(0)$ and the symmetric $d\times d$-matrix $\mathrm{Hess}\lambda(0)$ are related to the mean vector $\E_{\pi,0}[Y_1]$ and the asymptotic covariance matrix associated with the sequence $(Y_n - n\,\E_{\pi,0}[Y_1])/\sqrt{n}$. 
\begin{pro} \label{deri-hess} ~
\begin{enumerate}[(i)] \setlength{\itemsep}{1mm}
\item  If (\ref{AS3d}) holds for some $\alpha>1$,   
then $\nabla\lambda(0) = i\, \E_{\pi,0}[Y_1]$.  
\item If (\ref{AS3d}) holds for some $\alpha>2$, then  the following limit exists in the set of the non-negative symmetric $d\times d$-matrices:  
$$\Sigma \ : =\, \lim_n\frac{1}{n}\, \E_{\pi,0}\left[\, \big(Y_n - n\, \E_{\pi,0}[Y_1]\big) \, \big(Y_n - n\, \E_{\pi,0}[Y_1]\right)^*\, \big] = -\mathrm{Hess}\lambda(0).$$  
\end{enumerate}
\end{pro}
\begin{proof}{} Assume that $d=1$ for the sake of simplicity (for $d\geq2$, the proof is similar by using partial derivatives). By differentiating at $\zeta=0$ the equality $\E_{\pi,0}[e^{i\zeta Y_n}] = \lambda(\zeta)^n\, L(\zeta,1_{\X}) + R_n(\zeta,1_{\X})$ of Proposition~\ref{carac-en-0-disc}, we obtain: $$i\, \E_{\pi,0}[Y_n] = n\, \lambda^{(1)}(0) + L^{(1)}(0,1_{\X}) + R_n^{(1)}(0,1_{\X}).$$ 
Since $\E_{\pi,0}[Y_n] = n\, \E_{\pi,0}[Y_1]$ (from Corollary~\ref{Cor_statio}), we deduce  that $\lambda^{(1)}(0) = i$ and $\lim_n \E_{\pi,0}[Y_n]/n = i\, \E_{\pi,0}[Y_1]$ from (\ref{R-discret}). To prove (ii), assume for convenience that $\E_{\pi,0}[Y_1]=0$. Then $ \lambda^{(1)}(0) = 0$, and differentiating twice the above equality at $\zeta=0$ gives: $-\E_{\pi,0}[Y_n^2] = n\, \lambda^{(2)}(0) + L^{(2)}(0,1_{\X}) + R_n^{(2)}(0,1_{\X})$. We obtain the desired property by using again (\ref{R-discret}). 
\end{proof}

\subsection{Refinements of the CLT for discrete-time MAPs} \label{SS_DT}

In this subsection,  $(X_n,Y_n)_{n\in \N}$ is a MAP taking values in $\X\times\R^d$, with a driving Markov chain $(X_n)_{n\in \N}$  satisfying (AS\ref{AS1}-AS\ref{AS-spectralgap}).  
The assumptions below imply that $\E_{\pi,0}[|Y_1|] < \infty$, and for convenience we suppose that $\E_{\pi,0}[Y_1] = 0$ (if not, replace $Y_n$ by $Y_n - n\E_{\pi,0}[Y_1]$). 

Theorems~\ref{llt-discret} to \ref{ed-exp-discret} below have been established in \cite{HerPen09} for additive components of the form $Y_n = \sum_{k=1}^n \xi(X_k)$. To the best of our knowledge, the present extensions to general MAP are new.

\subsubsection{A local limit theorem} \label{llt-discrete}

The  classical Markov nonlattice condition is needed to state the local limit theorem (LLT): 
\begin{trivlist}
\item \textbf{Nonlattice condition.} {\it  There is no $a\in\R^d$, no closed subgroup $H$ in $\R^d$, $H\neq \R^d$, and no bounded measurable function $\beta\, :\, \X\r\R^d$ such that:  
$\, Y_1 + \beta(X_1) - \beta(X_0)\in a + H\ \ \P_{\pi,0}-a.s.$ }
\end{trivlist}
This condition is equivalent to the following operator-type property. For each $p\in(1,\infty)$ and for all compact subset $K$ of $\R^d\setminus\{0\}$, there exists $\rho\in(0,1)$ such that: 
\begin{equation} \label{non-lat-spectral}
\sup_{\zeta\in K}\|S_1(\zeta)^n\|_p = O(\rho^n).
\end{equation}
This result is established  in \cite[Sect.~5]{HerPen09} for additive functionals. The proof for general MAPs is similar. Since $\E_{\pi,0}[e^{i\, \langle \zeta , Y_n\rangle}] = \pi(S_1(\zeta)^n1_{\X})$ by (\ref{semi-group}), it follows that 
$$\sup_{\zeta\in K} \big|\E_{\pi,0}[e^{i\, \langle \zeta, Y_n\rangle}]\big| = O(\rho^n).$$
\begin{thm} \label{llt-discret} 
The assumptions of Theorem~\ref{clt-discret-multi} are supposed to be satisfied, so that $(Y_n/\sqrt{n})_{n\in\N^*}$ converges in distribution to a $d$-dimensional Gaussian vector with  covariance matrix $\Sigma$. Let us assume that $\Sigma$ is a definite positive matrix. Finally, suppose that the nonlattice condition is true. Then, we have for all compactly supported continuous function $g : \R^d\r\R$: 
$$\lim_{n\r+\infty} \sqrt{\det\Sigma}\, (2\pi n)^{\frac{d}{2}}\, \E_{\pi,0}[g(Y_n)\, ] = \int_{\R^d} g(x)dx.$$
\end{thm}
\begin{proof}{}
Thanks to (\ref{egalite_fcara_casdiscret}) with $f:=1_{\X}$, Theorem~\ref{llt-discret} can be established as in the i.i.d.~case: use Proposition~\ref{carac-en-0-disc} to control $L(\cdot,1_\X)$ and $R_n(\zeta,1_\X)$ and, as in \cite{Bre93},  use the nonlattice condition and the following second-order Taylor expansion  of $\lambda(\cdot)$, which follows from Theorem~\ref{clt-discret-multi} and from  \cite[Lem.~4.2]{Her05}: 
\begin{lem} \label{Dl-2-val-propre}
Assume that Conditions (AS\ref{AS-spectralgap}) and (\ref{AS3d}) with $\alpha=2$ hold and that $\E_{\pi,0}[Y_1]=0$. Then the function $\lambda(\cdot)$ in Equality~(\ref{egalite_fcara_casdiscret}) satisfies the following second-order Taylor expansion\footnote{A direct application of Proposition~\ref{carac-en-0-disc} gives this expansion, but under Condition (\ref{AS3d}) with $\alpha>2$.}: 
$$\lambda(\zeta) = 1 - \langle \zeta,\Sigma\zeta\rangle/2 + o(|\zeta|^2).$$ 
\end{lem}
\end{proof}

\begin{rem}
We mention that a local limit theorem has been obtained in \cite{MaxWoo97} for the process $(Y_n:=\sum_{k=1}^n Z_k)_{n\in\N^*}$ associated with a stationary hidden Markov chain $(X_n,Z_n)_{n\in\N}$. In \cite{MaxWoo97}, $(X_n)_{n\in\N}$ is only assumed to be an ergodic stationary Markov chain so that the additional conditions for the local limit theorem to hold are more involved than those of Theorem~\ref{llt-discret}. 
\end{rem}

\subsubsection{Rate of convergence in the one-dimensional CLT}

Here we suppose that $d=1$. Under the condition $\E_{\pi,0}[|Y_1|^{2+\varepsilon}] < \infty$, the asymptotic variance $\sigma^2$ of Proposition~\ref{deri-hess} is defined by $\sigma^2 := \lim_n \E_{\pi,0}[Y_n^2]/n$. 
\begin{thm} \label{u-b-e-gene} 
Under Conditions (AS\ref{AS1}-AS\ref{AS-spectralgap}) and (\ref{AS3d}) for some $\alpha>3$ and if $\sigma^2>0$, then there exists some constant $B>0$ such that  
\begin{equation} \label{B-E-discret}
\forall n\geq1,\ \ \ \sup_{a\in\R}\left|\, \P_{\pi,0}\bigg\{\frac{Y_n}{\sigma\sqrt n}\leq a\bigg\} - \Phi(a)\, \right| \leq \frac{B}{\sqrt n}
\end{equation}
where $\Phi(\cdot)$ is the distribution function of the Gaussian distribution $\cN(0,1)$. 
\end{thm}
\begin{proof}{}
Here, the functions $\lambda(\cdot)$, $L(\cdot):=L(\cdot,1_{\X})$ and $R_n(\cdot):=R_n(\cdot,1_{\X})$ in Proposition~\ref{carac-en-0-disc} are three times continuously differentiable on ${\cal O}$ and satisfy the following properties:
\begin{enumerate}
\item[] $\sup_{u\in {\cal O}}|L(u)-1|/|u| < \infty$ (from (\ref{L-discret}) and $L(0)=1$) 
\item[] $\sup_{u\in {\cal O}}\, |R_n(u)/u| =  O(\kappa^{n})$ (from (\ref{R-discret}) and $R_n(0)=0$)
\item[] $\lambda(u) = 1 -  \sigma^2u^2/2 + O(u^3)$ for $u$ small enough (since $\lambda^{(1)}(0) = 0$  and $\lambda^{(2)}(0) = -\sigma^2$ from Proposition~\ref{deri-hess}). 
\end{enumerate}
Then, we can borrow the proof of the Berry-Esseen theorem of the i.i.d.~case (see \cite{Fel71}). 
\end{proof}

\begin{rem}
The details of the previous proof are reported in \cite[Th.2]{HerLedPat09} for the additive functional $Y_n=\sum_{k=1}^n \xi(X_k,X_{k-1})$ of a  $V$-geometrically Markov chain. They are the same in our context. In fact, by writing  out the arguments of \cite[Th.~2]{HerLedPat09}, we can derive the following more precise property: the constant $B$ in (\ref{B-E-discret}) depends on the sequence $(Y_n)_{n\in\N}$, but only through $\sigma^2$ and $\E_{\pi,0}[|Y_1|^{3+\varepsilon}]$. Of course, this control is not as precise as in the i.i.d.~case \cite{Fel71}, but it is enough to obtain interesting statistical properties as in \cite{HerLedPat09,debora-note} or in Section~\ref{sect-B-E-statist}.
\end{rem}
\begin{rem} \label{u-b-e-xi} 
Let us consider the specific case $Y_n-Y_{n-1} = \xi(X_n)$ for some real-valued measurable function $\xi$. Under Conditions (AS\ref{AS1}-AS\ref{AS-spectralgap}), if the real number $\sigma^2$ defined in Remark~\ref{clt-discret-xi} is positive, 
then we have (\ref{B-E-discret}) under the expected moment condition $\pi(|\xi|^3)<\infty$.
This follows from \cite[Cor.~3.1]{Her08} which is based on the spectral method and martingale difference arguments (see also \cite[Sect.~6]{HerPen09}). Note that the moment condition $\pi(|\xi|^3)<\infty$ is optimal according to the i.i.d.~case \cite{Fel71}.   
\end{rem}
\begin{rem}
Let $(X_n)_{n\in\N}$ be a $\rho$-mixing Markov chain. 
The additive functionals of $(X_n)_{n\in\N}$ involved in the $M$-estimation of Markov models (see (\ref{mnf})) are  of the form $Y_n = \sum_{k=1}^n \xi(X_{k-1},X_k)$.
Since $(X_n,Y_n)_{n \in\N}$ is a MAP, Theorem~\ref{u-b-e-gene} applies provided that $\xi : \X\times \X \r \R$ is a measurable function such that $\E_{\pi,0}[\xi(X_0,X_1)] = 0$ and 
$\E_{\pi,0}[\, |\xi(X_0,X_1)|^{3+\varepsilon}\, ] < \infty$ for some $\varepsilon>0$. 
This will be supported by the statistical  result of Section~\ref{sect-B-E-statist}.
\end{rem}

Finally let us state a first-order Edgeworth expansion.
\begin{thm} \label{ed-exp-discret} 
Assume that Conditions (AS\ref{AS1}-AS\ref{AS-spectralgap}) and (\ref{AS3d}) hold for some $\alpha>3$, that $\sigma^2$ is positive and the nonlattice condition is true. Then, there exists $\mu_3\in\R$ such that: 
\begin{equation} \label{formule-ed-stat}
\P_{\pi,0}\bigg\{\frac{Y_n}{\sigma\sqrt n} \leq a\bigg\} = \Phi(a) + \frac{\mu_3}{6\sigma^3\sqrt n} (1-a^2)\, \eta(a) + o\left(\frac{1}{\sqrt n}\right)
\end{equation}
where $\eta(\cdot)$ is the density of the Gaussian distribution $\cN(0,1)$.
\end{thm}

Other limit theorems can be stated under Condition~(AS\ref{AS-spectralgap}) as, for instance, a multidimensional Berry-Esseen theorem in the Prohorov metric (see  \cite[Sect.~9]{HerPen09}), and the multidimensional renewal theorems (see \cite{GuiHer09}). Although  Proposition~\ref{carac-en-0-disc} extends to the case when the  order of regularity $m_0$ is not integer, it  does not allow to deal with the convergence of $Y_n$ (properly normalized) to stable laws, since we assume $\alpha > m_0$ (in place of the expected condition $\alpha=m_0$). For an additive functional $Y_n = \sum_{k=1}^n\xi(X_k)$, a careful examination of the proof of Lemmas~\ref{cont-noyau-discrete} and \ref{deri-noyau-discrete} shows that this limitation could be overcame under a condition  of the type : $\xi \in \L^{\beta}(\pi) \Longrightarrow P \xi \in \L^{\beta'}(\pi)$ with $\beta' > \beta$. Anyway mention that, under Condition~(AS\ref{AS-spectralgap}) and the previous condition on $\xi$, convergence to stable laws is obtained in \cite[Section 2.3]{JarKomOll09} by using a ``martingale approximation'' approach. A natural question is to ask wether the last condition on $\xi$ is necessary.  

\subsection{The non-stationary case}  \label{non-stat-carac} 
Under (AS\ref{AS-spectralgap}), we discuss the extension of the previous results to the non-stationary case. Let $\mu$ be the initial distribution of $(X_n)_{n\in\N}$. The careful use of \cite[Prop.~7.3]{HerPen09} allows us to extend Proposition~\ref{carac-en-0-disc} as follows. Under condition (\ref{AS3d})
\footnote{In this non-stationary case, we only require condition (\ref{AS3d}) with the stationary distribution $\pi$ and the mean vector remains $\E_{\pi,0}[Y_1]$.} 
with $\alpha > m_0$, and under the following assumption on $\mu$
\begin{description}
\item[(NS)]
$\mu$ is a bounded linear form on $\L^r(\pi)$ with  $r$ such that $1<r<\alpha s/(\alpha+m_0s)$,
\end{description} 
where $s>\alpha/(\alpha-m_0)$, 
all the conclusions of Proposition~\ref{carac-en-0-disc} remain true when $\pi$ is replaced by $\mu$, namely: for some bounded open neighborhood ${\cal O}$ of $\zeta=0$ in $\R^d$, we have for $f\in\L^{s}(\pi)$
\begin{equation}\label{egalite_fcara_mu}
\forall n\in\N,\ \forall \zeta\in{\cal O},\ \ \ \E_{\mu,0}\big[e^{i\, \langle \zeta , Y_n \rangle}\, f(X_n)\big] = \lambda(\zeta)^n\, L(\zeta,f,\mu) + R_n(\zeta,f,\mu),
\end{equation}
with $\C$-valued functions $\lambda(\cdot)$, $L(\cdot,f,\mu)$, $R_n(\cdot,f,\mu)$ satisfying the same properties as in Proposition~\ref{carac-en-0-disc}. It is worth noticing that $\lambda(\cdot)$ is the same function as in (\ref{egalite_fcara_casdiscret}), contrary to $L(\cdot,f,\mu)$ and $R_n(\cdot,f,\mu)$ which  both depend on $\mu$. 

Condition (NS) means that $\mu$ is absolutely continuous with respect to $\pi$ 
with density $\phi\in\L^{r'}(\pi)$ where  $r' = r/(r-1)$ is the conjugate number of $r$. It is easily checked that $r'> \alpha s/\big((\alpha-m_0) s-\alpha\big)>1$. Note that the bigger is the exponent $\alpha$ in  Condition (\ref{AS3d}), the closer to 1 is the allowed value of $r'$. 

Proposition~\ref{deri-hess} extends to the non-stationary case as follows.  
\begin{enumerate}[(i)]
	\item If (\ref{AS3d}) and (NS) hold with $m_0=1$, then $\nabla\lambda(0) = i\, \lim_n\E_{\mu,0}[Y_n]/n$.
	\item  If (\ref{AS3d}) and (NS) hold with $m_0=2$, then the conclusions of Proposition~\ref{deri-hess}(ii) remain true with $\mu$ in place of $\pi$. 
\end{enumerate}
Using the decomposition (\ref{egalite_fcara_mu}) (with $f:=1_{\X}$), we obtain as in the stationary case the following statements. 
\begin{enumerate}
	\item Under (AS\ref{AS-spectralgap}), (\ref{AS3d}) with $\alpha=2$ and $\mu$ satisfying condition (NS) with $m_0=1$:  the CLT, and the LLT under the additional non-lattice condition.
	\item Under (AS\ref{AS-spectralgap}), (\ref{AS3d}) with $\alpha>3$ and $\mu$ satisfying condition (NS) with $m_0=3$: the Berry-Esseen bound, and under the non-lattice condition, the first order Edgeworth expansion (\ref{formule-ed-stat}) with the additional term 
$-b_\mu\eta(u)/(\sigma\sqrt n)$, where $b_\mu$ is the  asymptotic bias: $b_\mu = \lim_n\E_{\mu,0}[Y_n]$ (see \cite{HerPen09} for details). 
\end{enumerate}

For instance, let us sketch the proof of the CLT. Equality (\ref{egalite_fcara_mu}) with  $f:=1_{\X}$ gives 
\begin{equation*} 
 \E_{\mu,0}\big[e^{i\, \langle \zeta , Y_n/\sqrt{n} \rangle}\big] = \lambda(\zeta/\sqrt{n})^n\, L(\zeta/\sqrt{n},1_{\X},\mu) + R_n(\zeta/\sqrt{n},1_{\X},\mu).
\end{equation*}
Since $m_0=1$ we have $\lim_n L(\zeta/\sqrt n,1_\X,\mu) = 1$ and $\lim_n R_n(\zeta/\sqrt n,1_\X,\mu) = 0$. Finally, 
the second-order Taylor expansion of Lemma~\ref{Dl-2-val-propre} shows that $\lim_n\lambda(\zeta/\sqrt n) = \exp(- \langle \zeta,\Sigma\zeta\rangle/2)$. 

In general, the previous statements 1. and 2. do not apply to the case when the initial distribution $\mu$ is a Dirac mass (which is not defined on $\L^r(\pi)$). However, when the state space $\X$ of the driving Markov chain is discrete, these statements are valid with any initial distribution $\delta_x$ provided that $\pi(x)>0$ (because $\delta_x$ is then a continuous linear form on each $\L^p(\pi)\equiv\ell^p(\pi)$). 
\subsection{The continuous-time case} \label{SS_CT}

In this section, we consider the case where $\T = (0,+\infty)$. The process $(X_t)_{t>0}$ is assumed to satisfy Conditions (AS\ref{AS1}-AS\ref{AS-spectralgap}). Let us mention that the moment condition (AS\ref{AS-alpha}) reduces to
	\[\forall v\in (0,1], \quad \E_{\pi,0}\big[|Y_v|^{\alpha}\big] <\infty
\]
when the semigroup $(Q_t)_{t\ge 0}$ is strongly continuous on $\L^2((\pi,0))$ (so is $(P_t)_{t> 0}$ on $\L^2(\pi)$). 

All the theorems of the previous subsection are extended to $(Y_t)_{t>0}$. Recall that Theorems~\ref{llt-discret} to \ref{ed-exp-discret} concern the multidimensional local limit theorem, the  one-dimensional Berry-Esseen theorem, the one-dimensional first-order Edgeworth expansion respectively. For the sake of simplicity, we still assume that $\E_{\pi,0}[Y_1] = 0$. 
\begin{thm} \label{ext-temps-continu}
The conclusions of Theorems~\ref{llt-discret} to \ref{ed-exp-discret} are valid for  $(Y_t/\sqrt{t})_{t>0}$ under the same assumptions, up to the following change: the moment condition (\ref{AS3d}) is reinforced (with the same condition on $\alpha$) in (AS\ref{AS-alpha}):  
$$\sup_{v\in(0,1]}\E_{\pi,0}\big[|Y_v|^{\alpha}\big] < \infty.$$
\end{thm}

Note that  the extensions to the non-stationary case presented in Subsection~\ref{non-stat-carac} can be adapted to  the continuous-time case. 

 When $Y_t$ is defined by $Y_t := \int_{0}^{t} \xi(X_s) \, ds$, any moment condition of the type  $\sup_{v\in[0,1]}\E_{\pi,0}\big[|Y_v|^{\alpha}\big] < \infty$ ($\alpha\geq 1$) is fulfilled if we have $\pi(|\xi|^\alpha) < \infty$. Indeed: 
$$\forall v\in[0,1],\quad \E_{\pi,0}\big[|Y_v|^{\alpha}\big] \leq \E_{\pi,0}\bigg[\int_{0}^{1} |\xi(X_s)|^{\alpha} \, ds\bigg] = \int_{0}^{1} \E_{\pi,0}\big[|\xi(X_s)|^{\alpha}\big] \, ds = \pi(|\xi|^\alpha).$$

 Note that the nonlattice condition used in Theorem~\ref{ext-temps-continu} is the same as  in the discrete-time case (see Subsection~\ref{llt-discrete}) and plays the same role. Indeed, writing  $t=n+v$ where $n$ is the integer part of $t$, we know that $\E_{\pi,0}[e^{i\, \langle \zeta , Y_t\rangle}] = \pi\big(S_1(\zeta)^n(S_v(\zeta)1_{\X})\big)$. 
Using (\ref{non-lat-spectral}) and the fact that $S_v$ is a contraction on $\L^p(\pi)$ ($p\in(1,+\infty)$), it follows that 
\begin{equation} \label{63}
\sup_{\zeta\in K} \big|\E_{\pi,0}\big[e^{i\, \langle \zeta, Y_t\rangle}\big]\big| = O(\rho^n).	
\end{equation}

We prove Proposition~\ref{carac-en-0-cont} below which is the continuous-time version of Proposition~\ref{carac-en-0-disc}. Then, combining Proposition~\ref{carac-en-0-cont} with relation (\ref{63}), the Fourier techniques of the i.i.d.~case can be used to extend Theorems~\ref{llt-discret}-\ref{ed-exp-discret} to $(Y_t/\sqrt{t})_{t>0}$
\begin{pro} \label{carac-en-0-cont}
Let $m_0\in\N^*$. Write time $t$ as $t=n+v$ where $n$ is the integer part of $t$. Under condition (AS\ref{AS-alpha}) for some 
$\alpha>m_0$, there exists a bounded open neighborhood ${\cal O}$ of $\zeta=0$ in $\R^d$ such that we have for all $f\in\L^{s}(\pi)$ with any $s>\alpha/(\alpha-m_0)$: 
$$\forall t\in(0,+\infty),\ \forall \zeta\in{\cal O},\ \ \ \E_{\pi,0}\big[e^{i\, \langle \zeta,Y_t\rangle}\, f(X_t)\big] = \lambda(\zeta)^n\, L\big(\zeta,S_v(\zeta)f\big) + R_n\big(\zeta,S_v(\zeta)f\big),$$
where $\lambda(\cdot)$, $L(\cdot,\cdot)$ and $R_n(\cdot,\cdot)$ are the functions of Proposition~\ref{carac-en-0-disc}. 
Moreover, the $\C$-valued functions $L_{v,f}(\zeta) := L\big(\zeta,S_v(\zeta)f\big)$ and $R_{n,v,f}(\zeta) := R_n\big(\zeta,S_v(\zeta)f\big)$ are of class $\cC^{m_0}$ on ${\cal O}$, and we have the following properties for $\ell=0,\ldots,m_0$:  
\begin{eqnarray*}
&\ & \sup_{\zeta\in{\cal O},\, v\in[0,1]}\, |L_{v,f}^{(\ell)}(\zeta)|<\infty \\  
&\ &  \exists \kappa\in(0,1),\ \sup_{\zeta\in{\cal O},\, v\in[0,1]}\, |R_{n,v,f}^{(\ell)}(\zeta)| = O(\kappa^n). 
\end{eqnarray*}
Note that we have $\lambda(0)=1$, $L_{v,f}(0) = \pi(f)$, and $R_{n,v,1_{\X}}(0) = 0$. 
\end{pro}
\begin{proof}{} From (\ref{Fourier}) and (\ref{semi-group}), we obtain for any $\zeta\in\R^d$, $f\in\L^p\ $ ($1\le p\le \infty$): 
\begin{equation} \label{egalite_fcara_cascontinu}
\E_{\pi,0}\big[e^{i\, \langle\zeta , Y_t\rangle}\, f(X_t)\big] = \pi\big(S_{n+v}(\zeta)f\big) = \pi\big(S_1(\zeta)^n\big(S_v(\zeta)f\big)\big) =  \E_{\pi,0}\big[e^{i\, \langle \zeta , Y_n \rangle}\, \big(S_v(\zeta)f\big)(X_n)\big],
\end{equation} 
and the desired expansion then follows from Proposition~\ref{carac-en-0-disc}. 
The two following (straightforward) extensions of Lemmas~\ref{cont-noyau-discrete}-\ref{deri-noyau-discrete} are needed to establish the others assertions. Let $t\in(0,+\infty)$. 
\begin{lem} \label{cont-noyau-cont} 
If $1\leq p' < p$, then the map $\zeta\mapsto S_t(\zeta)$ is  continuous from $\R^d$ to $\cL(\L^{p},\L^{p'})$.
\end{lem} 
\begin{lem} \label{deri-noyau-cont}
Assume that $\E_{\pi,0}[|Y_t|^{\alpha}] < \infty$ for some $\alpha>m_0$ ($m_0\in\N^*$), and let $1\leq j \leq m_0$. If $p >1 $ and $p_j :=\alpha p/(\alpha+jp) \ge 1 $, then $\zeta\mapsto S_t(\zeta)$ is $j$-times continuously differentiable from $\R^d$ to $\cL(\L^{p},\L^{p_j})$, and  $\sup_{\zeta\in\R^d}\|S_t^{(j)}(\zeta)\|_{p,p_j} \leq \E_{\pi,0}[|Y_t|^{\alpha}]^{j/ \alpha}$. 
\end{lem} 
\noindent The regularity properties (in $\zeta$) of the functions $L\big(\zeta,S_v(\zeta)f\big)$ and $R_n\big(\zeta,S_v(\zeta)f\big)$ are not a direct consequence of those stated in Proposition~\ref{carac-en-0-disc} because of the additional term $S_v(\zeta)f$. To that effect we need a careful use of the operator-type derivation procedure.  This part is postponed in Appendix~\ref{A} on the basis of \cite{HerPen09}.   
\end{proof}

\section{A Berry-Esseen theorem for the $M$-estimators of $\rho$-mixing Markov chains} \label{sect-B-E-statist}

The $M$-estimators are a general class of estimators in parametric statistics. This covers the special cases of maximum likelihood estimators, the least square estimators and the minimum contrast estimators. In the i.i.d.~case, a modern treatment on $M$-estimation is reported in  \cite[Chap.~5]{Vaa98},  and a Berry-Esseen bound for $M$-estimators is obtained in \cite{Pfa71}. In a statistical framework, such a bound has to be uniform in the parameters.  Pfanzagl's method, which is applied to Markov data in \cite{HerLedPat09}, requires a preliminary result on the rate of convergence in the CLT for additive functionals, with a precise control of the constants with respect to the functional (cf Remark~\ref{RemAppliBE}). Earlier extensions of \cite{Pfa71} to the Markov context are discussed in \cite{HerLedPat09}.
For $\rho$-mixing Markov chains, the closest work to ours is \cite{Rao73}. Our main improvement is on the moment conditions which are now close to those of the i.i.d.~case. A detailed comparison is presented at the end of the section.
 
 Let $\Theta$ be any nonempty parameter set. For a Markov chain $(X_n)_{n\in \N}$ with state space $\X$ and transition kernel $P_\theta$ which depends on $\theta\in\Theta$, we introduce the uniform $\L^2(\pi)$-spectral gap (i.e. the uniform $\rho$-mixing) property.
\begin{description}
\item[\textbf{(\cM)}] \textit{The Markov chain $(X_n)_{n\in \N}$ has a uniform $\L^2(\pi)$-spectral gap with respect to the  parameter set $\Theta$ if }
\begin{enumerate}
	\item \textit{for all $\theta \in\Theta$, $(X_n)_{n\in \N}$ has a unique $P_\theta$-invariant distribution $\pi_\theta$};
	\item \textit{for all $\theta \in\Theta$, $(X_n)_{n\in \N}$ is stationary (i.e.~$X_0\sim\pi_\theta$)}; 
	\item \textit{its transition kernel satisfies Condition (AS\ref{AS-spectralgap}) in a uniform way with respect to $\theta$, namely  there exist $C>0$ and $\kappa\in(0,1)$ such that 
$$\forall\theta\in\Theta,\ \forall n\geq 1,\ \ \ \|P_\theta^n-\Pi_\theta\|_2 \leq C\, \kappa^n,$$
where $\Pi_\theta(f) := \pi_\theta(f)\, 1_{\X}\, $ for $f\in\L^2(\pi)$}
\end{enumerate}
\end{description}
In order to derive a Berry-Esseen bound for the $M$-estimators of $(X_n)_{n\in \N}$ satisfying (\cM), we need a uniform Berry-Esseen bound for some specific additive functionals of the Markov chain $(X_n)_{n\in \N}$. In the next subsection, we propose such a uniform Berry-Esseen bound for the second component of a general parametric MAP. This result will be applied to the MAPs associated with these  specific additive functionals (see Remark~\ref{RemAppliBE}). 

\subsection{A uniform Berry-Essen bound for the second component of a parametric MAP}

Here we propose a refinement of Proposition~\ref{deri-hess} and Theorem~\ref{u-b-e-gene}. 
Let us introduce the following condition. 
\begin{itemize}
\item[\textbf{(A)}] {\it For every $\theta\in\Theta$, $(X_n,Y_n)_{n\in \N}$ is a $\X\times\R$-valued MAP, $Y_1$ is $\P_\theta$-integrable and centered (i.e.~$\E_\theta[Y_1] = 0$).}  
\end{itemize}
Below, the driving Markov chain $(X_n)_{n\in\N}$ is assumed to satisfy condition $(\cM)$. Thus, the  notation $\P_\theta$ stands for the underlying probability measure, which depends on $\theta$ through the transition kernel $Q_\theta$ of $(Y_n,X_n)_{n\in\N}$ and the initial (stationary) distribution $(\pi_\theta,0)$. $\E_\theta[\cdot]$ denotes the associated expectation. 
\begin{thm} \label{u-b-e-para}
Assume that Condition (A) is true for the MAP $(X_n,Y_n)_{n\in\N}$ and that the driving Markov chain $(X_n)_{n\in\N}$ satisfies Condition (\cM). If $M_2 := \sup_{\theta\in\Theta} \E_\theta\big[|Y_1|^{2+\varepsilon}\big] < \infty$ with some $\varepsilon>0$, then  $\sigma^2(\theta) := \lim_n \mathbb{E}_{\theta}[Y_n^2]/n$ is well-defined and is finite for each $\theta\in\Theta$, the function $\sigma^2(\cdot)$ is bounded on $\Theta$, and there exists a positive constant $C_Y$ such that
\begin{equation} \label{SigmaCY}
\forall n\geq1,\qquad
\sup_{\theta\in\Theta} \left|\sigma^2(\theta) - \frac{\mathbb{E}_{\theta}[Y_n^2]}{n}\right| \leq \frac{C_Y}{n}.
\end{equation}
The constant $C_Y$ depends on the sequence $(Y_n)_{n\in\N}$, but only through the constant $M_2$.

If the two following additional conditions hold true 
\begin{eqnarray}
&\ & \exists\, \varepsilon>0,\ \ M_3 := \sup_{\theta\in\Theta} \E_\theta\big[|Y_1|^{3+\varepsilon}\big] < \infty \label{M-Y} \\ 
&\ & \sigma_0 := \inf_{\theta\in\Theta}\sigma(\theta) > 0, \label{sigma-Y}
\end{eqnarray}
then there exists a positive constant $B_Y$ such that 
\begin{equation} \label{BE-BY}
\forall\theta\in\Theta,\ \forall n\geq1,\quad \sup_{a\in\R}\left|\, 
\P_\theta\bigg\{\frac{Y_n}{\sigma(\theta)\sqrt n}\leq a\bigg\} - \Phi(a)\, \right| \leq \frac{B_Y}{\sqrt n}.
\end{equation}
The constant $B_Y$ depends on the sequence $(Y_n)_{n\in\N}$ but only through $\sigma_0$ and the constant $M_3$.
\end{thm}

Recall that the proofs of Proposition~\ref{deri-hess} and Theorem~\ref{u-b-e-gene} are based on Proposition~\ref{carac-en-0-disc}. Here,  for any fixed $\theta\in\Theta$, Proposition~\ref{carac-en-0-disc} applies and gives an expansion of 
$\E_\theta[e^{i\, \zeta Y_n}\, f(X_n)]$, but for a neighbourhood ${\cal O}_\theta$ of $\zeta=0$, some $\C$-valued functions $\lambda_\theta(\cdot)$, $L_\theta(\cdot,f)$, $R_{\theta,n}(\cdot,f)$, and some $\kappa_\theta\in(0,1)$, which all may depend  on $\theta$.  Consequently, in order to prove Theorem~\ref{u-b-e-para}, we must establish that, under Conditions (\cM), (A) and the moment condition $\sup_{\theta\in\Theta} \E_\theta\big[|Y_1|^{m_0+\varepsilon}\big] < \infty$, all the conclusions of Proposition~\ref{carac-en-0-disc} are fulfilled in a uniform way with respect to $\theta\in\Theta$. 
This job has been done in \cite[Sect.~III.2]{HerLedPat09} in the context of $V$-geometrically ergodic Markov chains. The arguments in the present setting are the same up to the following changes: replace the uniform $V$-geometrical ergodicity assumption of \cite{HerLedPat09} by Assumption~($\cM$), and replace the domination condition ($D_{m_0}$) of \cite{HerLedPat09} by the moment condition $\sup_{\theta\in\Theta} \E_\theta\big[|Y_1|^{m_0+\varepsilon}\big] < \infty$.  The previous assumptions allow us to extend Lemmas \ref{cont-noyau-discrete}-\ref{deri-noyau-discrete}, and so Proposition~\ref{carac-en-0-disc}, in a uniform way in $\theta\in\Theta$. 

\begin{rem} \label{RemAppliBE} In the next subsection, Theorem~\ref{u-b-e-para} will be applied as follows. Given a Markov chain $(X_n)_{n\in\N}$ satisfying Condition ($\cM$) with respect to $\Theta$, we consider the MAP $(X_n,Y_n(p))_{n\in\N}$ where $Y_n(p)$ depends on some parameter $p\in\cP$ and is of the form 
$$Y_{n}(p) := \sum_{k=1}^n g(p, X_{k-1},X_k ).$$
 The property of the constant $C_Y$ in Theorem~\ref{u-b-e-para} ensures that Inequality (\ref{SigmaCY}) is  uniform in $p$ and $\theta$ when $M_2 := \sup_{ p\in\cP,\theta\in\Theta}\E_{\theta} \big[|Y_1(p)  |^{2+\varepsilon}\big]< \infty$ (of course, the asymptotic variance in (\ref{SigmaCY}) is replaced by some  $\sigma^2(\theta,p)$). In the same way, the Berry-Esseen bound (\ref{BE-BY}) is  uniform in $p$ and $\theta$ 
when $M_3 := \sup_{ p\in\cP,\theta\in\Theta}\E_{\theta} \big[|Y_1(p)  |^{3+\varepsilon}\big]< \infty$ and $\inf_{p\in\cP,\theta\in\Theta}\sigma(\theta,p)>0$.

Note that these comments extend to a general MAP $(X_n,Y_n)_{n\in\N}$ which may depend on some parameter $\gamma$ via its probability distribution and its functional form, provided that the bounds $M_2,M_3,  \sigma_0$ in Theorem~\ref{u-b-e-para} are uniform in $\gamma$.
\end{rem}

\begin{rem} \label{non-stat-b-e-theta}
The conclusions of Theorem~\ref{u-b-e-para} are also valid when 
$X_0\sim \mu_\theta$ with $ \mu_\theta$ of the form $\mu_\theta = \phi_\theta\, d\pi_\theta$, provided that $\sup_{\theta\in\Theta}\|\phi_\theta\|_{r'} < \infty$, with $r'$ defined as in Subsection~\ref{non-stat-carac} (case $m_0=3$).   
\end{rem}

\subsection{A Berry-Esseen bound for the $M$-estimators of $\rho$-mixing Markov chains}

Throughout this subsection, $\Theta$ is some general parameter space and $(X_n)_{n\ge 0}$ is a Markov chain with state space $\X$ satisfying the uniform $\L^2(\pi)$-spectral gap condition~($\cM$). The underlying probability measure and the associated expectation are denoted by $\P_\theta$ and $\E_\theta[\cdot]$. Recall that $(X_n)_{n\in\N}$ is assumed to be stationary under ($\cM$). Let us introduce the additive functional of $(X_n)_{n\ge 0}$
\begin{equation} \label{mnf}
M_n(\alpha) = \frac{1}{n} \sum_{k=1}^n F(\alpha,X_{k-1},X_k)
\end{equation}
where $\alpha\equiv\alpha(\theta)\in\mathcal{A}$ is the parameter of interest, $F(\cdot, \cdot, \cdot)$ is a real-valued measurable function on $\mathcal{A}\times \X^2$ and $\mathcal{A}$ is an open interval on the real line. Function $F$ is assumed to satisfy the following moment condition 
\begin{equation} \label{sup-F} 
\sup\big\{\,\E_\theta\big[|F(\alpha,X_{0},X_1)|\big],\ \theta\in\Theta,\ \alpha\in\mathcal{A}\,\big\} < \infty.
\end{equation}
Set $M_\theta(\alpha) := \E_{\theta}[F(\alpha,X_0,X_1)]$. We assume that, for
each $\theta\in\Theta$, there exists a unique $\alpha_0 =
\alpha_0(\theta)\in\mathcal{A}$, the so-called true value of the parameter of interest, such that we have $M_\theta(\alpha) > M_\theta(\alpha_0)$, $\forall \alpha\neq \alpha_0$.
To estimate $\alpha_0$, we consider the $M$-estimator $\widehat\alpha_n$  defined by 
$$M_n(\widehat\alpha_n ) =  \min_{\alpha\in\mathcal{A}} M_n(\alpha).$$
Also assume that, for all $(x,y)\in \X^2$, 
the map $\alpha\mapsto F(\alpha,x,y)$ is twice continuously
differentiable on $\mathcal{A}$. Let $F^{(1)}$ and $F^{(2)}$
be the first and second order partial derivatives of $F$ with respect to $\alpha$. Then 
\begin{equation}\label{mn_mn}
M^{(1)}_n(\alpha) = \frac{1}{n} \sum_{k=1}^n F^{(1)}(\alpha,X_{k-1},X_k),
\qquad M^{(2)}_n(\alpha) = \frac{1}{n} \sum_{k=1}^n
F^{(2)}(\alpha,X_{k-1},X_k).
\end{equation}

We shall appeal to the following assumptions.  
\textit{
\begin{description}
\item[(V0)]  There exists some real constant $\varepsilon > 0$ such that 
$$\sup_{\theta\in\Theta,\ \alpha\in\mathcal{A}}\E_\theta\left[\big|F^{(1)}(\alpha,X_{0},X_1)\big|^{3+\varepsilon} + \big|F^{(2)}(\alpha,X_{0},X_1)\big|^{3+\varepsilon}\ \right] < \infty. $$
\item[(V1)] $\ \displaystyle \forall \theta\in\Theta,\ \ \E_{\theta}[F^{(1)}(\alpha_0,X_0,X_1)] = 0$ and $\alpha_0\equiv\alpha_0(\theta)$ is the unique parameter value for which this property is true;
\item[(V2)] $\ m(\theta) := \E_{\theta}[F^{(2)}(\alpha_0,X_0,X_1)]$ satisfies
$\displaystyle  \inf_{\theta\in\Theta} m(\theta) >0 $;
\item[(V3)] $\forall n\geq 1$, $M^{(1)}_n(\widehat{\alpha}_n)=0$.  
   \end{description}}
 Notice that (V0) gives $\sup_{\theta\in\Theta} m(\theta) < \infty$. Set $Y^{(1)}_n(\alpha) := n\, M^{(1)}_n(\alpha)$ and $Y^{(2)}_n(\alpha) := n\, M^{(2)}_n(\alpha)$. Then, thanks to Theorem~\ref{u-b-e-para} applied to MAPs $(X_n,Y^{(1)}_n(\alpha))_{n\in \N}$ and $(X_n,Y^{(2)}_n(\alpha))_{n\in \N}$, the conditions (V0)-(V2) enable us to define the  asymptotic variances:
\begin{equation*}
\sigma_1^2(\theta) :=  \lim_n \frac{1}{n}\,
\mathbb{E}_{\theta}\big[Y^{(1)}_n(\alpha_0)^2\big] \qquad
\sigma_2^2(\theta)  :=  \lim_n \frac{1}{n}\,
\mathbb{E}_{\theta}\left[\big(Y^{(2)}_n(\alpha_0)- n\, m(\theta)\big)^2\right],
\end{equation*}
and we know that $\sup_{\theta\in\Theta}\sigma_j(\theta) < \infty$ for $j=1,2$. The following additional conditions are also required:\textit{
\begin{description}
    \item[(V4)] $\, \inf_{\theta\in\Theta} \sigma_j(\theta) > 0$
\textit{for} $j=1,2$.
\item[(V5)]  There exist $\eta>2$ and a measurable function $W>0$ such that  
$\sup_{\theta \in \Theta}\E_\theta[W^\eta]<\infty$ and
$$\forall (\alpha,\alpha')\in \cA^2, \ \forall (x,y) \in E^2, \ \ \ |F^{(2)}(\alpha,x,y)-F^{(2)}(\alpha',x,y)|\leq |\alpha-\alpha'| \ \big( W(x)+W(y) \big).$$
\item[(V6)]  There exists a sequence $\gamma_n \r 0$ such that
\[
\sup_{\theta\in\Theta}\P_{\theta}\big\{\,|\widehat{\alpha}_n-\alpha_0| \geq d\,\big\} \leq \gamma_n ,
\]
with $d:=\inf_{\theta \in \Theta}m(\theta)/\big(4 (\E_\theta[W(X_0)]+1)\big)$. 
\end{description}}
\begin{thm} \label{Th-B-E-estimator}
Assume that the Markov chain $(X_n)_{n\in\N}$ satisfies Condition (\cM),  that $F$ satisfies Condition~(\ref{sup-F}), that the $M-$estimator $\widehat\alpha_n$ is defined as above, and finally that Conditions \emph{(V0-V6)} are fulfilled. Set $\tau(\theta) :=
\sigma_1(\theta)/m(\theta)$. Then there exists a positive constant $C$ such that 
$$\forall n\geq 1,\quad \sup_{\theta\in\Theta}\sup_{u\in\R}\bigg|\P_{\theta}\left\{\frac{\sqrt n}{\tau(\theta)}\, (\widehat{\alpha}_n -\alpha_0) \leq u\right\}
- \Gamma(u)\bigg| \leq C \left( \frac{1}{\sqrt n} + \gamma_n \right)\, .$$
\end{thm}
Thanks to Theorem~\ref{u-b-e-para}, the proof of Theorem~\ref{Th-B-E-estimator} borrows the adaptation of Pfanzgal's method given in \cite{HerLedPat09}. One of the main difficulties in this method is to obtain a Berry-Esseen bound for the additive functionals  $Y_{n}(p):=\sum_{k=1}^n g(p,X_{k-1},X_k)$ with $p:=( v,q,\alpha_0)$ and 
$$g(p,X_{k-1},X_k) := F^{(1)}(\alpha_0,X_{k-1},X_k) + \frac{v}{\sqrt q}\frac{\sigma_1(\theta)}{m(\theta)}\big(F^{(2)}(\alpha_0,X_{k-1},X_k) - m(\theta)\big)$$
for $|v|\le 2 \sqrt{\ln q}$.  Observe that we have from (V0-V2): 
$$\sup_{\{(v,q): |v|\le 2 \sqrt{\ln q}\}, \, \theta\in\Theta}\E_\theta\big[\ \big|g(p,X_{0},X_1)\big|^{3+\varepsilon}\ \big] < \infty. $$ 
Then, Remark~\ref{RemAppliBE} gives the desired Berry-Esseen bound for $(Y_n(p))_{n\in\N}$ in a uniform way over the parameter $(\theta,p)$.

When the $X_n$'s are i.i.d., Theorem~\ref{Th-B-E-estimator} corresponds to Pfanzagl's theorem \cite{Pfa71} up to the following changes: in \cite{Pfa71}, $\pi_\theta$ is the common law of the $X_n$'s; the additive functional is $M_n(\alpha) = (1/n) \sum_{k=1}^n F(\alpha,X_k)$; we simply have $\sigma_1^2(\theta) = \E_{\theta}\big[F^{(1)}(\theta,X_0)^2\big]$ and $\sigma_2^2(\theta) = \E_{\theta}\big[(F^{(2)}(\theta,X_0)-m(\theta))^2\big]$, and finally Assumption~(V0) is replaced by the weaker (and optimal) moment condition: $\sup_{\theta\in\Theta}\E_{\theta}\big[\, |F^{(1)}(\theta,X_0)|^3 + |F^{(2)}(\theta,X_0)|^3\, \big] < \infty$. 

Earlier extensions of \cite{Pfa71} to the Markov context are discussed in \cite{HerLedPat09}. Let us  compare our result with that of \cite{Rao73}, in which the family of transition probabilities $P_\theta$ is assumed to satisfy a uniform Doeblin condition with respect to  $\theta\in\Theta$. This condition corresponds to a uniform $\L^\infty$-spectral gap condition with respect to $\Theta$ which is stronger than our Condition~($\cM$) (see Subsection~\ref{sec-basic-fact-map}). Let us mention 
that the moment condition on $F^{(1)}$ and $F^{(2)}$ in \cite{Rao73} is the following ($\alpha(\theta) = \theta$ in \cite{Rao73}): 
$$\sup_{x\in \X,\theta\in\Theta}\E_{\theta}\left[
\big|F^{(1)}(\theta,X_{0},X_1)\big|^3 + \big|F^{(2)}(\theta,X_{0},X_1)\big|^3
\, \big|\, X_0=x\right] < \infty.$$
Because of the supremum over $x\in \X$, this condition is in general much stronger than our moment condition~(V0) (despite the order $3+\varepsilon$ in (V0) instead of $3$). To see that, neglect the role of $\theta$ and consider a functional $f$ on $\X$. Then the difference between the condition used in \cite{Rao73} and (V0) is comparable to that between $\sup_{x\in \X}\E[|f(X_1)|^3\, |\, X_0=x]$ and $\E_\pi[|f(X_1)|^{3+\varepsilon}]$ (or, equivalently, between the supremum norm $\|P(|f|^3)\|_\infty$ and the norm $\|f\|_{3+\varepsilon}$ of $f$ in $\L^{3+\varepsilon}(\pi)$). Consequently, Theorem~\ref{Th-B-E-estimator} applies to the models considered in \cite{Rao73} but requires weaker moment conditions.
\begin{rem} \label{non-stat-estim}
The conclusion of Theorem~\ref{Th-B-E-estimator} holds true when $X_0\sim \mu_\theta$ and $\mu_\theta$  satisfies the condition given in Remark~\ref{non-stat-b-e-theta}. In this case, if $F$ is such that 
$$\sup_{\theta\in\Theta,\ \alpha\in\mathcal{A}}\E_\theta\big[|F(\alpha,X_{0},X_1)\big|^{1+\varepsilon}\big] < \infty$$ 
for some $\varepsilon > 0$, then $M_\theta(\alpha) = \E_\theta[F(\alpha,X_0,X_1)]$ can also be  defined by (see Subsection~\ref{non-stat-carac}): 
$$M_\theta(\alpha) = \lim_{n\rightarrow\infty}\E_{\theta,\mu_\theta} [M_n(\alpha)].$$ 
\end{rem}

\section{Conclusion}

In this paper, we propose limit theorems for the second component $(Y_t)_{t\in\T}$ of a discrete or continuous-time Markov Additive Process (MAP) $(X_t,Y_t)_{t\in\T}$ when $(X_t)_{t\in\T}$ has a $\L^2(\pi)$-spectral gap. The derivation of the CLT is based on a $\rho$-mixing condition strongly connected to the $\L^2(\pi)$-spectral gap property. The  results related to the convergence rate in the CLT are developed from  the weak spectral method of \cite{HerPen09}. Note that here the discrete and continuous-time cases are covered in a unified way. In this context, the  semigroup property (\ref{semi-group}) for the family of operators $\big(S_t(\zeta)\big)_{t\in\T}$ defined by $\big(S_t(\zeta)f\big)(x) := \E_{(x,0)}\big[e^{i\, \langle \zeta, Y_t\rangle}\, f(X_t)\big]\, $ ($\zeta\in\R^d$, $x\in X$) has a central role. We mention that this semigroup property is essentially true only for MAPs. The impact of the results is expected to be high for models involving a  $\L^2(\pi)$-spectral gap, since the limit theorems are valid for general (discrete and continuous time) MAPs, and under optimal (or almost optimal) moment conditions.   
This is illustrated in Section~\ref{sect-B-E-statist} where a Berry-Esseen bound for the $M$-estimator  associated with $\rho$-mixing Markov chains, is derived under the (almost) expected moment condition.

\appendix
\section{Additional material for the proof of Proposition \ref{carac-en-0-cont}.} \label{A}

Here, we study the regularity properties of the functions $\zeta \mapsto L\big(\zeta,S_v(\zeta)f\big)$ and $\zeta \mapsto R_n\big(\zeta,S_v(\zeta)f\big)$ involved in the decomposition of Proposition \ref{carac-en-0-cont}. 

\hspace*{5mm}\textbf{1)} Let us recall that we have (see (\ref{semi-gpe-discret})) 
$$\forall \zeta\in\R^d, \forall n\in\N,\quad \E_{\pi,0}\big[e^{i \langle \zeta , Y_n \rangle}\, f(X_n)\big] = \pi\big(S_1(\zeta)^n f\big).$$
and, for $t$ in some open neighbourhood $\cO$ of $\zeta=0$, (see \cite[7.2]{HerPen09})
$$S_1(\zeta)^n=\lambda(\zeta)^n\Pi(\zeta)+N(\zeta)^n,$$
where $\lambda(\zeta)$ is the dominant eigenvalue of $S_1(\zeta)$, $\Pi(\zeta)$ is the associated rank-one eigenprojection and $N(\zeta)$ is a bounded linear operator on each $\L^p(\pi)$  $1< p <\infty$. Both equalities imply that 
\begin{equation} \label{egalite_fcara}
\E_{\pi,0}\big[e^{i \langle \zeta , Y_n \rangle}\, f(X_n)\big] = \lambda(\zeta)^n\pi\big(\Pi(\zeta) f \big)+\pi\big(N(\zeta)^n f\big).
\end{equation} 
Furthermore the eigenprojection $\Pi(\zeta)$ and the operators $N(\zeta)^n$ are defined as in the standard perturbation theory by 
\[\Pi(\zeta) = \frac{1}{2i\pi} \oint_{\Gamma_1} \left(z-S_1(\zeta)\right)^{-1}\, dz,
\qquad
N(\zeta)^n = \frac{1}{2i\pi} \oint_{\Gamma_0} z^n \ \left(z-S_1(\zeta)\right)^{-1}\, dz, \]
where these line integrals are considered respectively on some oriented circle $\Gamma_1$ centered at $z=1$, and on some oriented circle $\Gamma_0$ centered at $z=0$, with radius~$\kappa<1$ where $\kappa$ is (for instance) $(1+\exp(-\varepsilon))/2$ with $\varepsilon$ defined in (\ref{Expo}).

\textbf{2)} Let us return to the continuous-time case. We obtain from (\ref{egalite_fcara_cascontinu}) and (\ref{egalite_fcara})  
\[\forall t\in\cO,\quad  \E_{\pi,0}\big[e^{i\, \langle\zeta , Y_t\rangle}\, f(X_t)\big] 
=\lambda(\zeta)^n\pi\big(\Pi(\zeta) (S_v(\zeta)f) \big)+\pi\big(N(\zeta)^n (S_v(\zeta)f)\big).\] 
Thus, we can write with the notations introduced in Proposition \ref{carac-en-0-cont} 
\[ \forall t\in\cO,\quad L(\zeta,S_v(\zeta)f):=\pi\big(\Pi(\zeta) (S_v(\zeta)f) \big), 
\qquad R_n(\zeta,S_v(\zeta)f):=\pi\big(N(\zeta)^n (S_v(\zeta)f)\big).
\]
Therefore, we only need to study the regularity of the map $\zeta \mapsto \left(z-S_1(\zeta)\right)^{-1} \circ S_v(\zeta)$ on $\cO$ for controlling that of the map $\zeta \mapsto \E_{\pi,0}\big[e^{i \langle \zeta , Y_t \rangle}\, f(X_t)\big]$ on $\cO$ (and, as a result, proving Proposition \ref{carac-en-0-cont}). 

\textbf{3)} Recall that $\|\cdot\|_{p,p'}$ denotes the operator norm in the space $\cL(\L^p,\L^{p'})$ of the linear bounded operators from $\L^p(\pi)$ to $\L^{p'}(\pi)$. The notation $W(\cdot) \in {\cC}^j(\theta,\theta')$ means that there exists a bounded open  neighborhood~${\cV}$ of $\zeta=0$ in $\R^d$ such that:
\begin{itemize}
	\item[] $\forall \zeta \in \cV$, $W(\zeta) \in {\cL}(\L^\theta,\L^{\theta'})$ and $W : {\cV} \mapsto {\cL}(\L^\theta,\L^{\theta'})$ has a continuous $j$-order differential on ${\cV}$.
\end{itemize}
Let us introduce the maps $U:\zeta \mapsto (z-S_1(\zeta))^{-1}$ and $V:\zeta\mapsto S_v(\zeta)$. 
We are going to apply the next obvious regularity property. Let $1\le \theta_{2m_0+2} < \theta_{2m_0+1} <\cdots< \theta_1 < \theta_0 <\infty$ (note that $\L^{\theta_0} \subset \L^{\theta_1} \subset \cdots \subset
\L^{\theta_{2{m_0}+1}}\subset \L^{\theta_{2{m_0}+2}}$), and  assume that we have: 
\begin{equation*}
\begin{split} 
& U\in \cC^0({\theta_{2{m_0}+1}},{\theta_{2{m_0}+2}})
\cap \cC^1({\theta_{2{m_0}-1}},{\theta_{2{m_0}+2}})  \cap \cdots
\cap \cC^{{m_0}-1}({\theta_3},{\theta_{2{m_0}+2}})
\cap \cC^{{m_0}}({\theta_1},{\theta_{2{m_0}+2}}) \\[0.2cm]
& V\in \cC^0({\theta_0},{\theta_1})
\cap \cC^1({\theta_1},{\theta_3})  \cap \cdots
\cap \cC^{{m_0}-1}({\theta_1},{\theta_{2{m_0}-1}})
\cap \cC^{{m_0}}({\theta_1},{\theta_{2{m_0}+1}}).
\end{split}
\end{equation*}
Then $UV \in \cC^{{m_0}}({\theta_0},{\theta_{2{m_0}+2}})$.

\textbf{4)} Let us introduce the following (non-increasing) maps from $[1,+\infty)$ to $\R$:  
\[ T_0(\theta):= \frac{\alpha\theta}{\alpha+\varepsilon_0\theta}  
 \quad \mbox{ and } \quad  T_1(\theta) := \frac{\alpha\theta}{\alpha+\theta}
\]
where $\varepsilon_0$ will be defined in (\ref{94}). 
Let $\theta>1$. Lemma~\ref{cont-noyau-cont} and the continuous inclusions between the Lebesgue spaces show that 
\begin{equation} \label{T0-cont} 
T_0(\theta)\ge 1\ \Rightarrow\ \forall\theta'\in[1,T_0(\theta)],\ S_v(\cdot)\in\cC^0(\theta,\theta').
\end{equation}
On the same way, Lemma~\ref{deri-noyau-cont} gives for $j=1,\ldots,m_0$: 
\begin{equation} \label{T1-reg} 
T_1^j(\theta)\ge 1\ \Longrightarrow\ \forall\theta'\in [1,T_1^j(\theta)],\ S_v(\cdot)\in\cC^j(\theta,\theta'),
\end{equation}
and the derivatives in the last property are uniformly bounded in $v\in[0,1]$ on any bounded open  neighborhood of $\zeta=0$. 

Now set $\theta_0 := s$, $\theta_1 := T_0(s)$, and observe that the assumption on $s$ (i.e.~$s>\alpha/(\alpha-m_0)$) is equivalent to $T_1^{m_0}(\theta_0) = \alpha \theta_0/(\alpha+m_0\theta_0) >1$, so that there exists $\varepsilon_0>0$ such that 
\begin{equation} \label{94}
(T_0T_1)^{m_0}T_0(\theta_1) = (T_0T_1)^{m_0}T_0\big(T_0(\theta_0)\big) = \frac{\alpha \theta_0}{\alpha+(m_0+(m_0+2)\varepsilon_0)\theta_0} =1.
\end{equation}
Define   
$$\theta_2:=T_0(\theta_1),\ \theta_3:=T_1T_0(\theta_1),\ \theta_4 := T_0T_1T_0(\theta_1),\ \ldots,\theta_{2m_0+2} := (T_0T_1)^{m_0}T_0(\theta_1),$$
namely: $\theta_{2j} := (T_0T_1)^{j-1}T_0(\theta_1)$  for $j=1,\ldots,m_0+1$, and $\theta_{2j+1} := T_1(T_0T_1)^{j-1}T_0(\theta_1)$ for $j=1,\ldots,m_0$. Note that $\theta_{2m_0+2}=1$. From (\ref{T0-cont})-(\ref{T1-reg}), $V(\cdot): = S_v(\cdot)$ satisfies the regularity properties stated in part~3), and the corresponding derivatives (on any bounded open  neighborhood of $\zeta=0$) are uniformly bounded in $v\in[0,1]$. 

Next, setting $I:=\{\theta_1,\theta_2,\ldots,\theta_{2m_0+2}\}$, it follows from (\ref{T0-cont})-(\ref{T1-reg}) (with $v=1$) 
that condition~$\cC(m_0)$ of \cite[7.1]{HerPen09} holds, so that the conclusions reported in \cite[p.48]{HerPen09} are true: 
\begin{itemize}
\item[$(H_0)$] {\it if $\theta\in I$ and 
$T_0(\theta)\in I$, then $\zeta \mapsto (z-S_1(\zeta))^{-1}\in\cC^0\left(\theta,T_0(\theta)\right)$ uniformly in $z \in \Gamma_0 \cup \Gamma_1$. } 
\end{itemize}
and for $\ell=1,\ldots,m_0$:
\begin{itemize}
\item[$(H_\ell)$]  
{\it If $\displaystyle \, \theta\in \bigcap_{k=0}^\ell \big[T_0^{-1}(T_0T_1)^{-k}(I)\cap 
    (T_1T_0)^{-k}(I)\big]$, then $\zeta \mapsto (z-S_1(\zeta))^{-1}\in \cC^\ell\left({\theta},
{(T_0T_1)^\ell T_0(\theta)}\right)$ uniformly in $z \in \Gamma_0 \cup \Gamma_1$.}
\end{itemize}
Therefore $U(\cdot) := (z-S_1(\cdot))^{-1}$ satisfies the regularity properties stated in part 3). 

\textbf{5)} Finally, we deduce from the property of part 3) that there exists a neighbourhood~${\cV}$ of $\zeta=0$ in $\R^d$ such that the map $\zeta \mapsto \left(z-S_1(\zeta)\right)^{-1} \circ S_v(\zeta)$ is $m_0$-times continuously differentiable from $\cV$ to 
${\mathcal L}(\L^s(\pi),\L^{1}(\pi))$ uniformly in $z \in \Gamma_0 \cup \Gamma_1$ 
and furthermore we have for $\ell=0,\ldots,m_0$: 
$$\sup \bigg \{\big\| \big(\left(z-S_1(\zeta)\right)^{-1}\circ S_v(\zeta) \big)^{(\ell)}\big\|_{s,1}; \ z \in \Gamma_0 \cup \Gamma_1, \ \zeta\in \cO,\ v\in(0,1] \bigg\} <\infty.$$  

\bibliographystyle{apalike}

\begin{thebibliography}{}

\end{thebibliography}


\begin{thebibliography}{}

\bibitem[Asmussen, 2000]{Asm00}
Asmussen, S. (2000).
\newblock {\em Ruin probabilities}.
\newblock World Sci. Publishing Co. Inc., River Edge, NJ.

\bibitem[Asmussen, 2003]{Asm03}
Asmussen, S. (2003).
\newblock {\em Applied probability and queues}, volume~51.
\newblock Springer-Verlag, NY, 2nd edition.

\bibitem[Asmussen et~al., 2004]{AsmAvrPis04}
Asmussen, S., Avram, F., and Pistorius, M.~R. (2004).
\newblock Russian and {A}merican put options under exponential phase-type
  {L}\'evy models.
\newblock {\em Stochastic Process. Appl.}, 109:79--111.

\bibitem[Babillot, 1988]{bab}
Babillot, M. (1988).
\newblock Th\'eorie du renouvellement pour des cha\^{\i}nes semi-markoviennes
  transientes.
\newblock {\em Ann. Inst. H. Poincar\'e Probab. Statist.}, 24:507--569.

\bibitem[Benveniste and Jacod, 1973]{BenJac73}
Benveniste, A. and Jacod, J. (1973).
\newblock Syst\`emes de {L}\'evy des processus de {M}arkov.
\newblock {\em Invent. Math.}, 21:183--198.

\bibitem[Bergh and L{\"o}fstr{\"o}m, 1976]{BerLof76}
Bergh, J. and L{\"o}fstr{\"o}m, J. (1976).
\newblock {\em Interpolation spaces. {A}n introduction}.
\newblock Springer-Verlag, Berlin.

\bibitem[Bhattacharya, 1982]{Bha82}
Bhattacharya, R.~N. (1982).
\newblock On the functional central limit theorem and the law of the iterated
  logarithm for {M}arkov processes.
\newblock {\em Probab. Theory Related Fields}, 60:185--201.

\bibitem[Billingsley, 1995]{Bil95}
Billingsley, P. (1995).
\newblock {\em Probability and measure}.
\newblock John Wiley \& Sons Inc., NY, 3th edition.

\bibitem[Bladt et~al., 2002]{BlaMeiNeuSer02}
Bladt, M., Meini, B., Neuts, M.~F., and Sericola, B. (2002).
\newblock Distributions of reward functions on continuous-time {M}arkov chains.
\newblock In {\em Matrix-analytic methods}, pages 39--62, Adelaide. World Sci.
  Publishing.

\bibitem[Bradley, 2005a]{Bra05}
Bradley, R.~C. (2005a).
\newblock Basic properties of strong mixing conditions. a survey and some open
  questions.
\newblock {\em Probab. Surv.}, 2:107--144.

\bibitem[Bradley, 2005b]{Bra05-1}
Bradley, R.~C. (2005b).
\newblock Introduction to strong mixing conditions ({V}olume {I}).
\newblock Technical report, Indiana University.

\bibitem[Breiman, 1993]{Bre93}
Breiman, L. (1993).
\newblock {\em Probability}.
\newblock SIAM.

\bibitem[Campanato, 1964]{Cam64}
Campanato, S. (1964).
\newblock Propriet\`a di una famiglia di spazi funzionali.
\newblock {\em Ann. Scuola Norm. Sup. Pisa}, 18:137--160.

\bibitem[Capp{\'e} et~al., 2005]{CapMouRyd05}
Capp{\'e}, O., Moulines, E., and Ryd{\'e}n, T. (2005).
\newblock {\em Inference in hidden {M}arkov models}.
\newblock Springer, NY.

\bibitem[\c{C}inlar, 1972]{Cin72-2}
\c{C}inlar, E. (1972).
\newblock {M}arkov additive processes {P}art {II}.
\newblock {\em Probab. Theory Related Fields}, 24:95--121.

\bibitem[\c{C}inlar, 1975]{Cin75}
\c{C}inlar, E. (1975).
\newblock {\em Introduction to stochastic processes.}
\newblock Prentice-Hall, Inc., Englewood Cliffs, New Jersey.

\bibitem[\c{C}inlar, 1977]{Cin77}
\c{C}inlar, E. (1977).
\newblock Shock and wear models and {M}arkov additive processes.
\newblock In {\em The theory and applications of reliability, with emphasis on
  {B}ayesian and nonparametric methods, {V}ol. {I}}, pages 193--214. Academic
  Press, NY.

\bibitem[Chen, 2004]{Che04}
Chen, M.-F. (2004).
\newblock {\em From {M}arkov chains to non-equilibrium particle systems}.
\newblock World Sci. Publishing Co. Inc., River Edge, NJ, 2nd edition.

\bibitem[Dehay and Yao, 2007]{DehYao07}
Dehay, D. and Yao, J.-F. (2007).
\newblock On likelihood estimation for discretely observed {M}arkov jump
  processes.
\newblock {\em Aust. N. Z. J. Stat.}, 49:93--107.

\bibitem[Doob, 1953]{Doo53}
Doob, J.~L. (1953).
\newblock {\em Stochastic processes}.
\newblock John Wiley \& Sons.

\bibitem[Ezhov and Skorohod, 1969a]{EzhSko69-1}
Ezhov, {\=I}.~{\=I}. and Skorohod, A.~V. (1969a).
\newblock Markov processes which are homogeneous in the second component. {I}.
\newblock {\em Theory Probab. Appl.}, 14:1--13.

\bibitem[Ezhov and Skorohod, 1969b]{EzhSko69-2}
Ezhov, {\=I}.~{\=I}. and Skorohod, A.~V. (1969b).
\newblock Markov processes which are homogeneous in the second component. {II}.
\newblock {\em Theory Probab. Appl.}, 14:652--667.

\bibitem[Feller, 1971]{Fel71}
Feller, W. (1971).
\newblock {\em An introduction to probability theory and its applications, Vol.
  {II}}.
\newblock John Wiley and Sons, NY.

\bibitem[Ferr\'e, 2010]{debora-note}
Ferr\'e, D. (2010).
\newblock D\'eveloppement d'{E}dgeworth d'ordre 1 pour des {M}-estimateurs dans
  le cas de cha\^ines {V}-g\'eom\'etriquement ergodiques.
\newblock {\em CRAS}, 348:331--334.

\bibitem[Fort et~al., 2003]{ForMouRobRos03}
Fort, G., Moulines, E., Roberts, G.~O., and Rosenthal, J.~S. (2003).
\newblock On the geometric ergodicity of hybrid samplers.
\newblock {\em J. Appl. Probab.}, 40:123--146.

\bibitem[Fuh and Lai, 2001]{FuhLai01}
Fuh, C.-D. and Lai, T.~L. (2001).
\newblock Asymptotic expansions in multidimensional {M}arkov renewal theory and
  first passage times for {M}arkov random walks.
\newblock {\em Adv. in Appl. Probab.}, 33:652--673.

\bibitem[Fukushima and Hitsuda, 1967]{FukHit67}
Fukushima, M. and Hitsuda, M. (1967).
\newblock On a class of {M}arkov processes taking values on lines and the
  central limit theorem.
\newblock {\em Nagoya Math. J.}, 30:47--56.

\bibitem[Ganidis et~al., 1999]{GanRoySim99}
Ganidis, H., Roynette, B., and Simonot, F. (1999).
\newblock Convergence rate of some semi-groups to their invariant probability.
\newblock {\em Stochastic Process. Appl.}, 79:243--263.

\bibitem[Genon-Catalot et~al., 2000]{GenJeaLar00}
Genon-Catalot, V., Jeantheau, T., and Lar{\'e}do, C. (2000).
\newblock Stochastic volatility models as hidden {M}arkov models and
  statistical applications.
\newblock {\em Bernoulli}, 6:1051--1079.

\bibitem[Glynn and Whitt, 1993]{GlyWhi93}
Glynn, P.~W. and Whitt, W. (1993).
\newblock Limit theorems for cumulative processes.
\newblock {\em Stochastic Process. Appl.}, 47:299--314.

\bibitem[Glynn and Whitt, 2002]{GlyWhi02}
Glynn, P.~W. and Whitt, W. (2002).
\newblock Necessary conditions in limit theorems for cumulative processes.
\newblock {\em Stochastic Process. Appl.}, 98:199--209.

\bibitem[Goldys and Maslowski, 2006a]{GolMas06-LN}
Goldys, B. and Maslowski, B. (2006a).
\newblock Exponential ergodicity for stochastic reaction-diffusion equations.
\newblock In {\em Stochastic partial differential equations and
  applications---{VII}}, pages 115--131. Chapman \& Hall/CRC, Boca Raton, FL.

\bibitem[Goldys and Maslowski, 2006b]{GolMas06}
Goldys, B. and Maslowski, B. (2006b).
\newblock Lower estimates of transition densities and bounds on exponential
  ergodicity for stochastic {PDE}'s.
\newblock {\em Ann. Probab.}, 34:1451--1496.

\bibitem[Gordin, 1978]{Gor78}
Gordin, M.~I. (1978).
\newblock On the central limit theorem for stationary {M}arkov processes.
\newblock {\em Soviet Math.~Dokl.}, 19:392--394.

\bibitem[Gou\"ezel, 2008]{Gou08}
Gou\"ezel, S. (2008).
\newblock Characterization of weak convergence of {B}irkhoff sums for
  {G}ibbs-{M}arkov maps.
\newblock Preprint.

\bibitem[Gou{\"e}zel and Liverani, 2006]{GouLiv06}
Gou{\"e}zel, S. and Liverani, C. (2006).
\newblock Banach spaces adapted to {A}nosov systems.
\newblock {\em Ergodic Theory Dynam. Systems}, 26:189--217.

\bibitem[Gravereaux and Ledoux, 2004]{GraLed04}
Gravereaux, J.-B. and Ledoux, J. (2004).
\newblock {P}oisson approximation for some point processes in reliability.
\newblock {\em Adv. in Appl. Probab.}, 36:455--470.

\bibitem[Grigorescu and Opri\c{c}an, 1976]{GriOpr76}
Grigorescu, S. and Opri\c{c}an, G. (1976).
\newblock Limit theorems for {$J-X$} processes with a general state space.
\newblock {\em Probab. Theory Related Fields}, 35:65--73.

\bibitem[Guibourg and Herv\'e, 2010]{GuiHer09}
Guibourg, D. and Herv\'e, L. (2010).
\newblock A renewal theorem for strongly ergodic {M}arkov chains in dimension
  $d\geq3$ and in the centered case.
\newblock {\em Potential Analysis}.
\newblock 10.1007/s11118-010-9200-2.

\bibitem[Guivarc'h, 1984]{guivarch}
Guivarc'h, Y. (1984).
\newblock Application d'un th\'eor\`eme limite local \`a la transcience et \`a
  la r\'ecurrence de marches al\'eatoires.
\newblock {\em Lecture Notes in Math.~Springer}, pages 301--332.

\bibitem[Guivarc'h, 2002]{Gui02}
Guivarc'h, Y. (2002).
\newblock Limit theorems for random walks and products of random matrices.
\newblock In {\em Proceedings of the CIMPA-TIFR School on Probability Measures
  on Groups, Mumbai 2002}, TIFR Studies in Mathematics series., pages 257--332.

\bibitem[Guivarc'h and Hardy, 1988]{GuiHar88}
Guivarc'h, Y. and Hardy, J. (1988).
\newblock Théorèmes limites pour une classe de chaînes de {M}arkov et
  applications aux difféomorphismes d'{A}nosov.
\newblock {\em Ann. Inst. H. Poincar\'e Probab. Statist.}, 24:73--98.

\bibitem[H{\"a}ggstr{\"o}m, 2006]{Hag06}
H{\"a}ggstr{\"o}m, O. (2006).
\newblock Acknowledgement of priority concerning ``{O}n the central limit
  theorem for geometrically ergodic {M}arkov chains''.
\newblock {\em Probab. Theory Related Fields}, 135:470.

\bibitem[Hennion and Herv\'e, 2001]{HenHer01}
Hennion, H. and Herv\'e, L. (2001).
\newblock {\em Limit theorems for {M}arkov chains and stochastic properties of
  dynamical systems by quasi-compactness}, volume 1766 of {\em Lecture Notes in
  Math.}
\newblock Springer.

\bibitem[Hennion and Herv{\'e}, 2004]{HenHer04}
Hennion, H. and Herv{\'e}, L. (2004).
\newblock Central limit theorems for iterated random {L}ipschitz mappings.
\newblock {\em Ann. Probab.}, 32(3A):1934--1984.

\bibitem[Herv\'e, 2005]{Her05}
Herv\'e, L. (2005).
\newblock Th\'eor\`eme local pour cha\^ines de {M}arkov de probabilit\'e de
  transition quasi-compacte. {A}pplications aux cha\^ines
  $v$-g\'eom\'etriquement ergodiques et aux mod\`eles it\'eratifs.
\newblock {\em Ann. Inst. H. Poincar\'e Probab. Statist.}, 41:179--196.

\bibitem[Herv\'e, 2008]{Her08}
Herv\'e, L. (2008).
\newblock Vitesse de convergence dans le th\'eor\`eme limite central pour des
  cha\^ines de {M}arkov fortement ergodiques.
\newblock {\em Ann. Inst. H. Poincar\'e Probab. Statist.}, 44:280--292.

\bibitem[Herv\'e et~al., 2009]{HerLedPat09}
Herv\'e, L., Ledoux, J., and Patilea, V. (2009).
\newblock A {B}erry-{E}sseen theorem on ${M}$-estimators for geometrically
  ergodic {M}arkov chains.
\newblock Accepted for publication in Bernoulli.

\bibitem[Herv\'e and P\`ene, 2010]{HerPen09}
Herv\'e, L. and P\`ene, F. (2010).
\newblock The {N}agaev-{G}uivarc'h method via the {K}eller-{L}iverani theorem.
\newblock {\em Bull. Soc. Math. France}, 138:415--489.

\bibitem[Hitsuda and Shimizu, 1970]{HitShi70}
Hitsuda, M. and Shimizu, A. (1970).
\newblock The central limit theorem for additive functionals of {M}arkov
  processes and the weak convergence to {W}iener measure.
\newblock {\em J. Math. Soc. Japan}, 22:551--566.

\bibitem[Holzmann, 2005]{Hol05}
Holzmann, H. (2005).
\newblock Martingale approximations for continuous-time and discrete-time
  stationary {M}arkov processes.
\newblock {\em Stochastic Process. Appl.}, 115:1518--1529.

\bibitem[Ibragimov, 1975]{Ibra75}
Ibragimov, I.~A. (1975).
\newblock A note on the central limit theorem for dependent random variables.
\newblock {\em Theorey Probab. Appl.}, 20:135--141.

\bibitem[Ibragimov and Linnik, 1971]{IbrLin71}
Ibragimov, I.~A. and Linnik, Y.~V. (1971).
\newblock {\em Independent and stationary sequences of random variables}.
\newblock Walters-Noordhoff, the Netherlands.

\bibitem[Jara et~al., 2009]{JarKomOll09}
Jara, M., Komorowski, T., and Olla, S. (2009).
\newblock Limit theorems for additive functionals of a {M}arkov chain.
\newblock {\em Ann. Applied Probab.}, 19:2270--2300.

\bibitem[Jarner and Hansen, 2000]{JarHan00}
Jarner, S.~F. and Hansen, E. (2000).
\newblock Geometric ergodicity of {M}etropolis algorithms.
\newblock {\em Stochastic Process. Appl.}, 85:341--361.

\bibitem[Jobert and Rogers, 2006]{JobRog06}
Jobert, A. and Rogers, L. C.~G. (2006).
\newblock Option pricing with {M}arkov-modulated dynamics.
\newblock {\em SIAM J. Control Optim.}, 44:2063--2078.

\bibitem[Jones, 2004]{Jon04}
Jones, G.~L. (2004).
\newblock On the {M}arkov chain central limit theorem.
\newblock {\em Probability Surveys}, 1:299--320.

\bibitem[Kartashov, 2000]{Kar00}
Kartashov, N.~V. (2000).
\newblock Determination of the spectral ergodicity exponent for the birth and
  death process.
\newblock {\em Ukrain. Math. J.}, 52:1018--1028.

\bibitem[Keilson and Wishart, 1964]{KeiWis64}
Keilson, J. and Wishart, D. M.~G. (1964).
\newblock A central limit theorem for processes defined on a finite {M}arkov
  chain.
\newblock {\em Proc. Cambridge Philos. Soc.}, 60:547--567.

\bibitem[Keller and Liverani, 1999]{KelLiv99}
Keller, G. and Liverani, C. (1999).
\newblock Stability of the spectrum for transfer operators.
\newblock {\em Annali della Scuola Normale Superiore di Pisa - Classe di
  Scienze Sér. 4}, XXVIII:141--152.

\bibitem[Kipnis and Varadhan, 1986]{KipVar86}
Kipnis, C. and Varadhan, S. R.~S. (1986).
\newblock Central limit theorem for additive functionals of reversible {M}arkov
  processes and applications to simple exclusions.
\newblock {\em Comm. Math. Phys.}, 104:1--19.

\bibitem[Lezaud, 2001]{Lez01}
Lezaud, P. (2001).
\newblock Chernoff and {B}erry-{E}sseen inequalities for {M}arkov processes.
\newblock {\em ESAIM: P\&S}, 5:183--201.

\bibitem[Liggett, 1989]{Lig89}
Liggett, T.~M. (1989).
\newblock Exponential {$L_2$} convergence of attractive reversible nearest
  particle systems.
\newblock {\em Ann. Probab.}, 17:403--432.

\bibitem[Limnios and Opri\c{c}an, 2001]{LimOpr01}
Limnios, N. and Opri\c{c}an, G. (2001).
\newblock {\em Semi-{M}arkov Processes and Reliability.}
\newblock Birkhauser Boston Inc.

\bibitem[Maigret, 1978]{Mai78}
Maigret, N. (1978).
\newblock Th\'eor\`eme de limite centrale fonctionnel pour une cha\^\i ne de
  {M}arkov r\'ecurrente au sens de {H}arris et positive.
\newblock {\em Ann. Inst. H. Poincar\'e Probab. Statist.}, 14:425--440.

\bibitem[Maxwell and Woodroofe, 1997]{MaxWoo97}
Maxwell, M. and Woodroofe, M. (1997).
\newblock A local limit theorem for hidden {M}arkov chains.
\newblock {\em Statist. Probab. Lett.}, 32:125--131.

\bibitem[Meyn and Tweedie, 1993]{MeyTwe93}
Meyn, S.~P. and Tweedie, R.~L. (1993).
\newblock {\em {M}arkov chains and stochastic stability}.
\newblock Springer Verlag.

\bibitem[Nagaev, 1957]{Nag57}
Nagaev, S.~V. (1957).
\newblock Some limit theorems for stationary {M}arkov chains.
\newblock {\em Theory Probab. Appl.}, 11:378--406.

\bibitem[Neveu, 1961]{Nev61}
Neveu, J. (1961).
\newblock Une généralisation des processus à accroissements positifs
  indépendants.
\newblock {\em Abh. Math. Sem. Univ. Hambourg}, 25:36--61.

\bibitem[{\"O}zekici and Soyer, 2004]{OzeSoy04}
{\"O}zekici, S. and Soyer, R. (2004).
\newblock Reliability modeling and analysis in random environments.
\newblock In {\em Mathematical reliability: an expository perspective}, pages
  249--273. Kluwer Acad. Publ., Boston, MA.

\bibitem[Pacheco and Prabhu, 1995]{PacPra95}
Pacheco, A. and Prabhu, N.~U. (1995).
\newblock Markov-additive processes of arrivals.
\newblock In {\em Advances in queueing}, pages 167--194. CRC, Boca Raton, FL.

\bibitem[Pacheco et~al., 2009]{PacTanPra09}
Pacheco, A., Tang, L.~C., and Prabhu, N.~U. (2009).
\newblock {\em Markov-modulated processes \& semiregenerative phenomena}.
\newblock World Sci. Publishing, NJ.

\bibitem[Peligrad, 1987]{Pel87}
Peligrad, M. (1987).
\newblock On the central limit theorem for {$\rho$}-mixing sequences of random
  variables.
\newblock {\em Ann. Probab.}, 15:1387--1394.

\bibitem[Pfanzagl, 1971]{Pfa71}
Pfanzagl, J. (1971).
\newblock The {B}erry-{E}sseen bound for minimum contrast estimates.
\newblock {\em Metrika}, 17:81--91.

\bibitem[Pinsky, 1968]{Pin68}
Pinsky, M. (1968).
\newblock Differential equations with a small parameter and the central limit
  theorem for functions defined on a finite {M}arkov chain.
\newblock {\em Probab. Theory Related Fields}, 9:101--111.

\bibitem[Rao, 1973]{Rao73}
Rao, B. L. S.~P. (1973).
\newblock On the rate of convergence of estimators for {M}arkov processes.
\newblock {\em Probab. Theory Related Fields}, 26:141--152.

\bibitem[Revuz and Yor, 1999]{RevYor99}
Revuz, D. and Yor, M. (1999).
\newblock {\em Continuous martingales and {B}rownian motion}.
\newblock Springer-Verlag, Berlin, 3th edition.

\bibitem[Roberts and Rosenthal, 1997]{RobRos97}
Roberts, G.~O. and Rosenthal, J.~S. (1997).
\newblock Geometric ergodicity and hybrid markov chains.
\newblock {\em Elect. Comm. in Probab.}, 2:13--25.

\bibitem[Roberts and Rosenthal, 2004]{RobRos04}
Roberts, G.~O. and Rosenthal, J.~S. (2004).
\newblock General state space {M}arkov chains and {MCMC} algorithms.
\newblock {\em Probab. Surv.}, 1:20--71.

\bibitem[Roberts and Tweedie, 2001]{RobTwe01}
Roberts, G.~O. and Tweedie, R.~L. (2001).
\newblock Geometric {$L^2$} and {$L^1$} convergence are equivalent for
  reversible {M}arkov chains.
\newblock {\em J. Appl. Probab.}, 38A:37--41.

\bibitem[Rosenblatt, 1971]{Ros71}
Rosenblatt, M. (1971).
\newblock {\em Markov processes. {S}tructure and asymptotic behavior}.
\newblock Springer-Verlag, New-York.

\bibitem[Stefanov, 2006]{Ste05}
Stefanov, V.~T. (2006).
\newblock Exact distributions for reward functions on semi-{M}arkov and
  {M}arkov additive processes.
\newblock {\em J. Appl. Probab.}, 43:1053--1065.

\bibitem[Steichen, 2001]{Ste01}
Steichen, J.~L. (2001).
\newblock A functional central limit theorem for {M}arkov additive processes
  with an application to the closed {L}u-{K}umar network.
\newblock {\em Stoch. Models}, 17:459--489.

\bibitem[Touati, 1983]{Tou83}
Touati, A. (1983).
\newblock Th\'eor\`emes de limite centrale fonctionnels pour les processus de
  {M}arkov.
\newblock {\em Ann. Inst. H. Poincar\'e Probab. Statist.}, 19:43--55.

\bibitem[van~der Vaart, 1998]{Vaa98}
van~der Vaart, A.~W. (1998).
\newblock {\em Asymptotic statistics}.
\newblock Cambridge Univ. Press.

\bibitem[Wu, 2004]{Wu04}
Wu, L. (2004).
\newblock Essential spectral radius for {M}arkov semigroups. {I}. {D}iscrete
  time case.
\newblock {\em Probab. Theory Related Fields}, 128(2):255--321.

\end{thebibliography}

\end{document}